\documentclass[reqno]{amsart}
\usepackage{amssymb,amsmath}
\usepackage{amsthm}
\usepackage{color,graphicx}
\usepackage{hyperref}

\newcommand{\R}{\mathbb R}
\newcommand{\N}{{\mathbb N}}

\newcommand{\Z}{\mathbb Z}

\newcommand{\sgn}{\text{sgn}}

\newtheorem{theorem}{Theorem}[section]
\newtheorem{proposition}[theorem]{Proposition}
\newtheorem{remark}[theorem]{Remark}
\newtheorem{lemma}[theorem]{Lemma}

\newtheorem{definition}[theorem]{Definition}

\begin{document}
\vglue-1cm \hskip1cm
\title[The Benjamin-Ono-Zakharov-Kuznetsov equation]{The IVP for the Benjamin-Ono-Zakharov-Kuznetsov equation in weighted Sobolev spaces}



\author[A. Cunha]{Alysson Cunha}
\address{IMECC-UNICAMP, Rua S\'ergio Buarque de Holanda, 651, 13083-859, Cam\-pi\-nas-SP, Bra\-zil.}
\email{apastor@ime.unicamp.br}

\author[A. Pastor]{Ademir Pastor}
\address{IMECC-UNICAMP, Rua S\'ergio Buarque de Holanda, 651, 13083-859, Cam\-pi\-nas-SP, Bra\-zil}
\email{ra030382@ime.unicamp.br}



\begin{abstract}
In this paper we study  the initial-value problem associated with the
Benjamin-Ono-Zakharov-Kuznetsov equation.  We prove that the IVP for such
equation is locally well-posed in the usual Sobolev spaces $H^{s}(\R^2),$
$s>2$, and in the anisotropic spaces $H^{s_1,s_2}(\R^2)$, $s_2>2$, $s_1\geq
s_2$. We also study the persistence properties of the solution and local
well-posedness in the weighted Sobolev class
$$
\mathcal{Z}_{s,r}=H^{s}(\R^{2})\cap L^{2}((1+x^{2} +y^{2})^rdxdy),
$$
where $s>2$, $r\geq 0$, and $s\geq 2r$. Unique continuation properties of the
solution are also established. These continuation principles show that our
persistence properties are sharp. Most of our arguments are accomplished
taking into account that ones for the Benjamin-Ono equation.
\end{abstract}

\maketitle

\section{Introduction}\label{introduction}

This paper is concerned with the initial-value problem (IVP) for the
Benjamin-Ono-Zakharov-Kuznetsov (BO-ZK) equation
\begin{equation}\label{bozk}
\begin{cases}
u_{t}+\mathcal{H}\partial_{x}^{2}u+u_{xyy}+uu_{x}=0, \;\;(x,y)\in\R^2, \;t>0, \\
u(x,y,0)=\phi(x,y),
\end{cases}
\end{equation}
where $u=u(x,y,t)$ is a real-valued function and $\mathcal{H}$ stands for the
Hilbert transform defined as
$$
\mathcal{H}u(x,y,t)=\mathrm{p.v.}\frac{1}{\pi}\int_{\R}\frac{u(z,y,t)}{x-z}dz.
$$
Recall that $\mathrm{p.v.}$ denotes the Cauchy principal value.

The BO-ZK equation was recently introduced in \cite{Jorge} and
\cite{Latorre}, and it has applications to electromigration in thin
nanoconductors on a dielectric substrate. It may also be viewed as a natural
two-dimensional generalization of the  Benjamin-Ono equation
\begin{equation}\label{boequation}
u_{t}+\mathcal{H}\partial_{x}^{2}u+uu_{x}=0, \;\;x\in\R, \;t>0.
\end{equation}

Throughout the paper,   well-posedness is understood  in Kato's sense, that
is, it includes existence, uniqueness, persistency property, and  continuous
dependence of the map data-solution.

Before stating our main theorems, let us recall some previous results
concerning the problem \eqref{bozk}. In \cite{EP} and \cite{EP1}, the authors
studied existence and stability of solitary waves solutions having the form
$u(x,y,t)=\varphi_c(x-ct,y)$, where $c$ is a real parameter and $\varphi_c$
is smooth and decays to zero at infinity. By using the variational approach
introduced by Cazenave and Lions \cite{CL}, they proved, in particular, the
orbital stability of ground state solutions. Unique continuation properties
was addressed in \cite{EP3}, where the authors showed if a sufficiently
smooth solution has support in a rectangle, for all $t$ as long as the
solution exists, then it must vanish everywhere.

The IVP \eqref{bozk} has similar features as the one for the BO equation
\eqref{boequation}. Indeed, following the ideas in \cite{MST}, the authors in
\cite{EP2} established the ill-posedness of \eqref{bozk}  in the sense that
it cannot be solved in the usual (anisotropic) $L^2$-based Sobolev space by
using a fixed point theorem. More precisely, for all $s\in\R$, the map
data-solution cannot be $C^2$-differentiable at the origin from $H^s(\R^2)$
to $H^s(\R^2)$.

Let us now turn attention to the results in the present paper. We start with
the following.

\begin{theorem}\label{120}
Let $s>2$. Then for any $\phi\in H^{s}(\R^2),$ there exist a positive
$T=T(\|\phi\|_{H^{s}})$  and a unique solution $u\in C([0,T];H^{s}(\R^2))$ of
the $\mathrm{IVP}$ \eqref{bozk}. Furthermore, the flow-map $\phi\mapsto u(t)$
is continuous in the $H^{s}$-norm and there exists a function $\rho\in
C([0,T];\R)$ such that
$$
\|u(t)\|^2_{H^s}\leq \rho(t), \qquad t\in [0,T].
$$
\end{theorem}

With respect to anisotropic Sobolev spaces, we have the following.

\begin{theorem}\label{105}
Let $\phi \in H^{s_1,s_2}(\R^2),$ where $s_2>2$ and $s_1\geq s_2$. Then there
exist
 $T=T(\|\phi\|_{s_1, s_2})$ and a unique solution $u\in C([0,T];H^{s_1,
s_2}(\R^2))$ of the $\mathrm{IVP}$ \eqref{bozk}. Furthermore, the flow-map
$\phi\mapsto u(t)$ is continuous in the $H^{s_1,s_2}$-norm and there exists a
function $\rho\in C([0,T];\R)$ such that
\begin{equation}\label{rhoestinteo}
\|u(t)\|^2_{s_1,s_2}\leq \rho(t), \qquad t\in [0,T].
\end{equation}
\end{theorem}

Theorems \ref{120} and \ref{105} are proved by using the {\em parabolic
regularization method}. Since their proofs are quite similar we only prove
Theorem \ref{105}. It should be noted that this technique does not rely on
the dispersive effects of the equation in question. Thus, improvements of the
above results should consider such effects.

In comparison with the BO equation \eqref{boequation}, many authors, in
\cite{BP}, \cite{IK}, \cite{MP}, and \cite{tao} for instance, using
appropriated gauge transformations, have obtained strong results of local
well-posedness in low regularity Sobolev spaces. In the case of BO-ZK
equation, it is not clear how to get a suitable transformation and we do not
know if such approach could be used to improve Theorems \ref{120} and
\ref{105}. On the other hand, we believe that the above theorems can be
improved by employing the techniques introduced by \cite{KK} and \cite{KT},
which combines Strichartz estimates with some energy estimates. This will
appear elsewhere.

Our main focus in this paper consists in proving persistence properties and
local well-posedness in weighted Sobolev spaces. The question we address is
the following: suppose we have an initial data in the Sobolev space
$H^s(\R^2)$ with some ``additional'' decay at infinity; is it true that the
solution inherits the same decay?

For the BO equation, this question has been addressed for instance in
\cite{GermanPonce}, \cite{FLP}, \cite{Iorio}, and  \cite{Iorio1} and the
answer produces very interesting results. In particular, there exist no
nontrivial solutions with strong decay.

 Our first result in this direction is concerned with the persistence and the local well-posedness in
 the weighted Sobolev spaces $H^s(w^2)$ (see notation below).

 \begin{theorem}\label{H}
Let $w$ be a smooth weight with all its first, second, and third derivatives
bounded. Then, the IVP \eqref{bozk} is locally well-posed in $H^s(w^2)$,
$s>2$.
 \end{theorem}

To prove Theorem \ref{H} we follow the arguments in \cite{Iorio1} and
\cite{Aniura1} with some adaptations to our problem. In \cite{Aniura1} the
author have considered a two-dimensional model which can also be viewed as a
generalization of \eqref{boequation}, but with a different structure. Note
that boundedness  is required only on its derivatives but not on $w$.

 Next, we have the following.

\begin{theorem}\label{B1}
The following statements are true.
\begin{itemize}
\item [i)] If $s>2$ and $r\in [0,1]$ then the $\mathrm{IVP}$ \eqref{bozk} is
locally well-posed in $\mathcal{Z}_{s,r}$. Furthermore, if $r\in (1,5/2)$ and
$s\geq 2r$ then  \eqref{bozk} is locally well-posed in $\mathcal{Z}_{s,r}.$
\item [ii)] If $r\in [5/2,7/2)$ and $s\geq 2r$, then the $\mathrm{IVP}$
\eqref{bozk} is locally well-posed in $\dot{\mathcal{Z}}_{s,r}.$
\end{itemize}
\end{theorem}

To the best of our knowledge, the study of the Cauchy problem in fractional
weighted Sobolev spaces  in the spirit of  Theorem \ref{B1} was initiated by
Fonseca and Ponce in \cite{GermanPonce}, where, in particular, the authors
proved the counterpart of Theorem \ref{B1} for the BO equation (see also
\cite{FLP},\cite{FLP1}, \cite{Iorio}, and \cite{Iorio1}). So, here we extend
their ideas to the two-dimensional case in order to establish our results. Of
course, since \eqref{bozk} includes a third-order derivative and the weights
in hand are two-dimensional, additional troubles are expected in comparison
with the BO equation. However, by performing suitable estimates we are able
to handle with all difficulties.

Note that Theorem \ref{B1} establishes some balancing between the regularity
and the decay rate of the initial data. In particular, the condition
$s\geq2r$ is necessary if $r>1$.

Theorem \ref{B1} is in some sense sharp, which is evidenced  by the following
unique continuation principles.

\begin{theorem}\label{P1}
Let $u\in C([0,T]; \mathcal{Z}_{4,2})$ be a solution of the $\mathrm{IVP}$
\eqref{bozk}. If there exist two different times $t_1, t_2 \in [0,T]$ such
that $u(t_j)\in \mathcal{Z}_{5,5/2}, \ j=1,2 \ $, then
$$
\hat{u}(0,\eta,t)=0
$$
for all $\eta\in\R$ and $t\in[0,T]$.
\end{theorem}

\begin{theorem}\label{P2}
Let $u\in C([0,T]; \mathcal{Z}_{4,2})$ be a solution of the $\mathrm{IVP}$
\eqref{bozk}. If there exist three different times $t_1, t_2, t_3 \in [0,T]$
such that $u(t_j)\in \mathcal{Z}_{7,7/2}, \ j=1,2,3, \ $, then
$$
u(x,y,t)=0
$$
for all $x,y\in\R$ and $t\in[0,T]$.
\end{theorem}

Two important conclusions emerge from Theorems \ref{P1} and \ref{P2}. The
first one is that the condition $\widehat{\phi}(0,\eta)=0$, for all
$\eta\in\R$, is necessary to have persistence property in
$\mathcal{Z}_{s,5/2}$, $s\geq5$. In particular, part (i) of Theorem \ref{B1}
shows to be sharp. The second one is that if an initial data $\phi$ has a
decay stronger than $|(x,y)|^{7/2}$ then the persistence property does not
hold, unless it vanishes identically. This shows that part (ii) of Theorem
\ref{B1} is also sharp. A similar conclusion for BO equation was obtained in
\cite[Theorems 2 and 3]{GermanPonce}.

It should also be  pointed out that our unique continuation statements are
stronger than the ones in \cite{EP3}, where it is assumed that the solution
has compact support for all $t\in[0,T]$.

\begin{remark}
We believe one can prove similar results as those above in the weighted
anisotropic Sobolev spaces $Z_{s_1,r_1}^{s_2,r_2}=H^{s_1,s_2}\cap
L^2_{r_1,r_2}$, where
$$
L^2_{r_1,r_2}=L^2((1+x^{2r_1}+y^{2r_2})dxdy).
$$
This is currently under investigation.
\end{remark}

The paper is organized as follows. In Section \ref{notation} we introduce the
notation used throughout the paper and give some preliminaries results. By
using the parabolic regularization method, we prove in Section \ref{localwp}
the local well-posedness in Sobolev spaces. Sections \ref{localweighted} and
\ref{localweighted1} are dedicated to prove Theorems \ref{H} and \ref{B1},
respectively. Finally, in Section \ref{uniquep} we establish Theorems
\ref{P1} and \ref{P2}.


\section{Notation and Preliminaries}\label{notation}

Let us first introduce some notation. We use $c$ to denote various constants
that may vary line by line; if necessary we use subscript to indicate
dependence on parameters. With $[A,B]$ we denote the commutator between the
operators $A$ and $B$. By $\|\cdot\|_p$ we denote the usual $L^p$ norm.
Because the $L^2$ norm appears frequently below, we use the notation
$\|\cdot\|$ for it. The scalar product in $L^2$ will be then represented by
$(\cdot,\cdot)$. In particular, note if $f=f(x,y)$ then
$\|f\|=\|\|f(\cdot,y)\|_{L^2_x}\|_{L^2_y}$, where by $\|\cdot\|_{L^2_z}$ we
mean the $L^2_z$ norm with respect to the variable $z$. The integral $\int f$
will stand, otherwise is stated, for the integration of $f$ over $\R^2$.

For any $s\in \R$, $H^s:=H^s(\R^2)$ represents the usual $L^2$-based Sobolev
space with norm $\|\cdot\|_{H^s}$. The Fourier transform of $f$ is defined as
$$
\hat{f}(\xi,\eta)=\int_{\R^2}e^{-i(x\xi+y\eta)}f(x,y)dxdy.
$$
Given any complex number $z$, let us define the operators $J^z_x$, $J^z_y$,
and $J^z$ via its Fourier transform by
$$
\widehat{J^z_xf}(\xi,\eta)=(1+\xi^2)^{z/2}\hat{f}(\xi,\eta);
$$
$$
\widehat{J^z_yf}(\xi,\eta)=(1+\eta^2)^{z/2}\hat{f}(\xi,\eta);
$$
$$
\widehat{J^zf}(\xi,\eta)=(1+\xi^2+\eta^2)^{z/2}\hat{f}(\xi,\eta).
$$
Given $s_1,s_2\in \R$, the anisotropic Sobolev space
$H^{s_1,s_2}=H^{s_1,s_2}(\R^2)$ is the set of all tempered distributions $f$
such that
$$
\|f\|^{2}_{s_1,s_2}=\|f\|^{2} +
\|J_x^{s_1}f\|^{2}+\|J_y^{s_2}f\|^{2}<\infty.
$$
The scalar product in $H^{s_1,s_2}$ will be denoted by
$(\cdot,\cdot)_{s_1,s_2}$.

As usual $\mathcal{S}(\R^2)$ will denote the Schwartz space. Given $s\in \R$
and $w:\R^2\to [0,\infty)$, we define the weighted Sobolev space to be
$$
H^{s}(w^2):=H^{s}(\R^2)\cap L^{2}(w^{2}dxdy).
$$
In particular, for $r>0$, we denote
$$
\mathcal{Z}_{s,r}:=H^s(\R^2)\cap L^2_r,
$$
where $L^2_r:=L^2(\langle x,y\rangle^{2r}dxdy)$. Here, $\langle
x,y\rangle:=(1+x^2+y^2)^{1/2}$. The norm in $\mathcal{Z}_{s,r}$ is given by
$\|\cdot\|_{\mathcal{Z}_{s,r}}^2=\|\cdot\|_{H^s}^2+\|\cdot\|_{L^2_r}^2$.
Also, the subspace $\dot{\mathcal{Z}}_{s,r}$ of $\mathcal{Z}_{s,r}$ is
defined as
$$
\dot{\mathcal{Z}}_{s,r}:=\{f\in\mathcal{Z}_{s,r}\ | \ \hat{f}(0,\eta)=0, \
\eta\in\R \}.
$$

Suppose  $\phi\in \mathcal{Z}_{s,r}$ and let $u$ be the local solution of
\eqref{bozk}. Assuming  that $u$ is sufficiently regular, we can integrate
the equation with respect to $x$ to obtain, at least formally,
\begin{equation}\label{consquan}
\int_{\R}u(x,y,t)dx=\int_\R\phi(x,y)dx, \quad y\in\R,
\end{equation}
as long as the solution exists. This implies that
\begin{equation}\label{fourieru}
\hat{u}(0,\eta,t)=\hat{\phi}(0,\eta), \quad \eta\in \R,
\end{equation}
for all $t$ for which the solution exists. In particular, if $\phi\in
\dot{\mathcal{Z}}_{s,r}$ then $\hat{u}(0,\eta,t)=0$,  for all $\eta\in\R$ and
$t$ for which the solution exits.

 Let $N\in \Z^{+}$. We define a function $\beta_N:\R\to\R$ by letting
\begin{eqnarray}
\beta_{N}(x):=\left\{\begin{array} {lccc}
\langle x \rangle \ \mathrm{if} \  |x|\leq N,\\
2N \ \mathrm{if} \ |x|\geq 3N,
\end{array} \right.
\end{eqnarray}
where $\langle x \rangle = (1+x^2)^{1/2}$. Also, we assume that $\beta_{N}$
is smooth, symmetric, non-decreasing in $|x|$ with $\beta_{N}'(x)\leq 1,$ for
any $x\geq 0$, and there exists a constant $c$ independent of $N$ such that
$|\beta_{N}''(x)|\leq c \partial_{x}^{2}\langle x \rangle.$ Now, we define
the two-dimensional truncated weights
\begin{equation}\label{defWN}
w_{N}(x,y):=\beta_{N}(r), \qquad \mbox{where} \;\;\;r=(x^2 +y^2)^{1/2}.
\end{equation}

Next, we introduce some preliminaries results which will be useful to
demonstrate our main results. Most of these results have appeared elsewhere,
but for the sake of completeness we bring then here.

\begin{definition}
We say that a non-negative function $w\in L^{1}_{loc}(\R)$ satisfies the
$A_{p}$ condition, with $1<p<\infty$, if
\begin{eqnarray}\label{Ap}
\sup_{Q \ \mathrm{interval}}\left (\frac{1}{|Q|}\int _{Q}w\right)\left (\frac{1}{|Q|}\int_{Q}w^{1-p'}\right)^{p-1}=c(w)<\infty,
\end{eqnarray}
where $1/p+1/p'=1.$
\end{definition}

Since our main results are concerned with weighted spaces, we need to deal
with the Hilbert transform in weighted spaces. The next result will be
sufficient to our purposes.

\begin{theorem}\label{boundhilbert}
The condition \eqref{Ap} is necessary  and sufficient for the boundedness of
the Hilbert transform $\mathcal{H}$ in $L^{p}(w(x)dx),$ i.e.,
\begin{eqnarray}\label{bound}
\left(\int_{-\infty}^{\infty}|\mathcal{H}f|^{p}w(x)dx\right)^{1/p}\leq c^{*} \left(\int_{-\infty}^{\infty}|f|^{p}w(x)dx\right)^{1/p}.
\end{eqnarray}
\end{theorem}
\begin{proof}
See \cite{Hunt&Munk}.
\end{proof}

\begin{remark}\label{condap}
It is not difficult to check that $|x|^{\alpha}$ satisfies the $A_{2}$
condition  if and only if $\alpha \in (-1,1).$ More generally, $|x|^\alpha$
satisfies the $A_{p}$ condition if and only if $\alpha\in(-1,p-1)$ (see also
\cite[page 441]{GermanPonce}).
\end{remark}

The next three results will be widely used in the proof of Theorem \ref{B1}.
\begin{theorem}\label{indc}
For $p\in [2,\infty)$ the inequality \eqref{bound} holds with $c^{*}\leq
c(p)c(w),$ where $c(p)$ depends only on $p$ and $c(w)$ is as in \eqref{Ap}.
Moreover, for $p=2$ this estimate is sharp.
\end{theorem}
\begin{proof}
See \cite{kaikina}.
\end{proof}

The next theorem  is a generalization of Calder\'on commutator estimate
\cite{Cald}. Its proof can be found in \cite{Dawson}. Moreover it has
applications for several dispersive models.

\begin{theorem}\label{Comu}
For any $p\in (1,\infty)$ and $l,m\in \Z^{+}\cup \{0\},$ $l+m\geq 1,$ there exists $c=c(p;l;m)>0$ such that
\begin{equation}\label{cald}
\|\partial_{x}^{l}[\mathcal{H};a]\partial_{x}^{m}f\|_{p}\leq c \|\partial_{x}^{l+m}a\|_{\infty}\|f\|_{p}.
\end{equation}
\end{theorem}

Let us recall that $L^{p}_{s}:=(1-\Delta)^{-s/2}L^{p}(\R^n)$. Such spaces
can be characterized by the following result.

\begin{theorem}\label{stein}
Let $b\in (0,1)$ and $2n/(n+2b)<p<\infty.$ Then $f\in L^{p}_{b}(\R^{n})$ if and only if
\begin{itemize}
\item [a)] $f\in L^{p}(\R^{n}),$
\item [b)]
$\mathcal{D}^{b}f(x)={\displaystyle \left (
\int_{\R^{n}}\frac{|f(x)-f(y)|^{2}}{|x-y|^{n+2b}}dy\right)^{1/2}} \in
L^{p}(\R^{n}),$

\noindent with
\begin{equation}\label{equiv}
\|f\|_{b,p}\equiv \|(1-\Delta)^{b/2}f\|_{p}=\|J^{b}f\|_{p}\simeq \|f\|_{p}+\|D^{b}f\|_{p}\simeq \|f\|_{p}+\|\mathcal{D}^{b}f\|_{p},
\end{equation}
where for $s\in \R$, $D^{s}=(-\Delta)^{s/2}$ with
$D^{s}=(\mathcal{H}\partial_{x})^{s}$ if $n=1.$
\end{itemize}
\end{theorem}
\begin{proof}
See \cite{Stein}.
\end{proof}

\begin{remark}
The operator $\mathcal{D}^b$ introduced in Theorem \ref{stein} is sometimes
referred to as the  Stein derivative of order $b$. The last equivalence in
\eqref{equiv} says that we can compute the norm in $L^{p}_{b}$ by using
either $D^b$ or $\mathcal{D}^{b}$. The advantage in using $\mathcal{D}^{b}$
is that we are able to do point estimates easily (see Proposition
\ref{Pontual} and Lemma \ref{P} below).
\end{remark}

From the previous theorem, part b), with $p=2$ and $b\in(0,1)$, we have
\begin{equation}\label{Leib}
\|\mathcal{D}^{b}(fg)\| \leq \|f\mathcal{D}^{b}g\| + \|g\mathcal{D}^{b}f\|.
\end{equation}

\begin{proposition}\label{Pontual}
Let $b\in (0,1)$. For any $t>0$ and $x\in\R$,
\begin{equation}
\mathcal{D}^{b}(e^{-itx|x|})\leq c(t^{b/2}+t^{b}|x|^{b}).
\end{equation}
\end{proposition}
\begin{proof}
See  \cite{NahasPonce}.
\end{proof}

We also have the following estimate.

\begin{lemma}\label{P}
Let $b\in (0,1),$ then for all $t>0$ and $x,\eta\in \R$,
$$
\mathcal{D}^{b}(e^{it\eta^{2}x})\leq c(b)\eta^{2b}t^{b},
$$
where $c(b)$ depends only on $b$.
\end{lemma}
\begin{proof}
First note that
\begin{equation*}
\begin{split}
\Big(\mathcal{D}^b(e^{it\eta^{2}x})\Big)^{2}=&\int_{\R}\frac{|e^{it\eta^{2}x}-e^{it\eta^{2}y}|^{2}}{|x-y|^{1+2b}}dy\\
=&\int_{\R}\frac{|e^{it\eta^{2}x}-e^{it\eta^{2}(x-y)}|^{2}}{|y|^{1+2b}}dy\\
=&\int_{\R}\frac{|1-e^{-it\eta^{2}y}|^{2}}{|y|^{1+2b}}dy\\
=&\ (\eta^{2}t)^{2b}\int_{\R}\frac{|1-e^{iy}|^{2}}{|y|^{1+2b}}dy\\
=&\
(\eta^{2}t)^{2b}\left(\int_{-1}^{1}\frac{|1-e^{iy}|^{2}}{|y|^{1+2b}}dy+\int_{|y|>1}
\frac{|1-e^{iy}|^{2}}{|y|^{1+2b}}dy\right).
\end{split}
\end{equation*}
From the inequality $|1-e^{iy}|\leq 2|y|$,  $y\in [-1,1]$, we have
$$
\int_{-1}^{1}\frac{|1-e^{iy}|^{2}}{|y|^{1+2b}}dy\leq 2\int_{-1}^{1}\frac{dy}{|y|^{2b-1}}=4\int_{0}^{1}\frac{dy}{y^{2b-1}}=\frac{2}{1-b}.
$$
Moreover,
$$\int_{|y|>1}\frac{dy}{|y|^{1+2b}}=2\int_{1}^{\infty}\frac{dy}{y^{1+2b}}=\frac{2}{b}.$$
Therefore,
$$\mathcal{D}^{b}(e^{it\eta^{2}x})\leq \left(\frac{2}{1-b}+\frac{2}{b}\right)^{1/2}(\eta^{2}t)^{b}.$$
This completes the proof of the lemma.
\end{proof}

The next proposition will be used in the proof of Theorems \ref{P1}  and
\ref{P2}.
\begin{proposition}\label{localint}
Let $p\in (1,\infty)$. If $f\in L^{p}(\R)$ is such that  there exists $x_0
\in \R$ for which $f(x_{0}^{+}),$ $f(x_{0}^{-})$ are defined and
$f(x_{0}^{+})\neq f(x_{0}^{-}),$ then for any $\delta>0,$
$\mathcal{D}^{1/p}f\notin L^{p}_{loc}(x_0-\delta,x_0+\delta)$ and
consequently $f\notin L^{p}_{1/p}(\R).$
\end{proposition}
\begin{proof}
See \cite{GermanPonce} (see also \cite{jose}).
\end{proof}

At last, we recall some results to be used in the proof  Theorems \ref{B1},
\ref{P1}, and \ref{P2}.

\begin{lemma}\label{inter}
Let $a,b>0.$ Assume that $J^{a}f=(1-\Delta)^{a/2}f\in L^{2}(\R^2)$ and
$\langle x,y \rangle^b f=(1+x^2 +y^2)^{b/2}f\in L^{2}(\R^2).$ Then for any
$\alpha \in (0,1)$
\begin{equation}\label{inter1}
\|J^{\alpha a}(\langle x, y \rangle^{(1-\alpha)b}f)\|\leq c\|\langle x, y
\rangle^{b}f\|^{1-\alpha}\|J^{a}f\|^{\alpha}.
\end{equation}
Moreover, the inequality \eqref{inter1} is  still valid with $w_{N}(x,y)$
instead of $\langle x, y \rangle$ with a constant $c$ independent of $N.$
\end{lemma}
\begin{proof}
The proof  is similar to that carried out in \cite[Lemma 1]{GermanPonce}.
\end{proof}

\begin{proposition}\label{C}
If $f\in L^{2}(\R)$ and $\phi \in H^{2}(\R),$ then
\begin{equation}
\|[D^{1/2};\phi]f\|_{L^{2}(\R)}\leq c\|\phi\|_{H^{2}(\R)}\|f\|_{L^{2}(\R)}.
\end{equation}
\end{proposition}


\section{Local well-posedness in Sobolev spaces}\label{localwp}

In this section, we are concerned with local well-posedness in (anisotropic)
Sobolev spaces. We only prove Theorem  \ref{105}. The ideas are by now quite
standard, so we only sketch the main steps. The arguments are based on that
proposed  in \cite{Iorio}. Let $\mu>0$ and consider the following
perturbation of \eqref{bozk}:
\begin{equation}\label{2}
\begin{cases}
u_{t}+\mathcal{H}\partial_{x}^{2}u+u_{xyy}+uu_{x}=\mu\Delta u , \;\;(x,y)\in\R^2, \;t>0, \\
u(x,y,0)=\phi(x,y).
\end{cases}
\end{equation}
Let us first consider the linear IVP
\begin{equation}\label{2.1a}
    \begin{cases}
{\displaystyle u_t+ \mathcal{H}\partial_x^2u+ u_{xyy}-\mu\Delta u =  0,  }  \qquad (x,y) \in \mathbb{R}^2, \,\,\,\, t>0, \\
{\displaystyle  u(x,y,0)=\phi(x,y)}.
\end{cases}
\end{equation}
The solution of \eqref{2.1a} is given by
\begin{equation}\label{2.2a}
u(t)=E_\mu(t)\phi(x,y)= \int_{\mathbb{R}^2} e^{i\big(t(-\xi
|\xi|+\xi\eta^2)+x\xi+y\eta\big)-t\mu(\xi^2+\eta^2)}\hat{\phi}(\xi,\eta) d\xi
d\eta.
\end{equation}

A straightforward calculation reveals the following.

\begin{proposition}\label{propa1}
Let $\lambda_1, \lambda_2\in[0,\infty)$ and $\mu>0.$ Then,
\begin{itemize}
    \item [a)] for any $t>0$ and $s_1,s_2\in \R$, $E_{\mu}(t)$ is a bounded
        operator from $H^{s_1,s_2}$ to $H^{s_1+\lambda_1,s_2
        +\lambda_2}$. Moreover
$$
\|E_{\mu}(t)\phi\|_{s_1+\lambda_1,s_2+\lambda_2}\leq
C_{\lambda_1,\lambda_2,\mu}
(1+t^{-\lambda_{1}/2}+t^{-\lambda_{2}/2})\|\phi\|_{s_1,s_2}, \quad \phi\in
H^{s_1,s_2},
$$
and the map $t\in (0,\infty)\longmapsto E_{\mu}(t)\phi\in
H^{s_1+\lambda_1,s_2 +\lambda_2}$ is continuous.

    \item [b)]  $E_{\mu}(t)$ is a semigroup of contractions in $H^{s_1,s_2}$ and  can
        be extended, when $\mu=0$, to a unitary group.
\end{itemize}
\end{proposition}
\begin{proof}
See for instance \cite[Theorem 2.1]{Iorio}.
\end{proof}

To proceed with the arguments, we need to use that $H^{s_1,s_2}$ is a Banach
algebra. So, we prove the following.

\begin{proposition}
Let $u,v \in H^{s_1,s_2}$, with $s_1,s_2>1$. Then
$$\|uv\|_{s_1,s_2}\leq c_{s_1s_2}\|u\|_{s_1,s_2}\|v\|_{s_1,s_2}.$$
\end{proposition}
\begin{proof}
It is easy to see that, for all $f\in H^{s_1,s_2}$,
\begin{equation}\label{embLinfty}
\|f\|_{\infty}\leq c_{s_1s_2}\|f\|_{s_1,s_2}.
\end{equation}
Now, fixing $y$ in Lemma X4 of \cite{KP}, we have
$$
\|J^{s_1}_x(uv)\|_{L^2_x}\leq
c(\|u\|_{L^{\infty}_{x}}\|J_{x}^{s_1}v\|_{L^{2}_{x}}+\|v\|_{L^{\infty}_{x}}\|J_{x}^{s_1}u\|_{L^{2}_{x}}).
$$
Now, taking the $L^2$ norm with respect to $y$, using Holder's inequality and
\eqref{embLinfty}, we deduce
\begin{equation}\label{2002}
\begin{split}
\|J_{x}^{s_1}(uv)\|=& \|\|J_{x}^{s_1}(uv)\|_{L^{2}_{x}}\|_{L^{2}_{y}}\\
 \leq & c(\|\|u\|_{L^{\infty}_{x}}\|J_{x}^{s_1}v\|_{L^{2}_{x}}+\|v\|_{L^{\infty}_{x}}\|J_{x}^{s_1}u\|_{L^{2}_{x}}\|_{L^{2}_{y}})\\
\leq& \ c(\|u\|_{L^{\infty}_{xy}}\|\|J_{x}^{s_1}v\|_{L^{2}_{x}}\|_{L^{2}_{y}}+\|v\|_{L^{\infty}_{xy}}\|\|J_{x}^{s_1}u\|_{L^{2}_{x}}\|_{L^{2}_{y}})\\
\leq & \ c_{s_1,s_2}(\|u\|_{s_1,s_2}\|J_{x}^{s_1}v\|+\|v\|_{s_1,s_2}\|J_{x}^{s_1}u\|)\\
\leq & \ c_{s_1,s_2} \|u\|_{s_1,s_2}\|v\|_{s_1,s_2}.
\end{split}
\end{equation}
Analogously,
\begin{eqnarray}\label{2003}
\|J_{y}^{s_2}(uv)\|\leq c_{s_1s_2} \|u\|_{s_1,s_2}\|v\|_{s_1,s_2}.
\end{eqnarray}
The result then follows from \eqref{2002} and \eqref{2003}.
\end{proof}

With these tools in hand, we can prove the local well-posedness of \eqref{2}.

\begin{theorem}\label{modtheorem}
Let $\mu>0,$ and $\phi\in H^{s_1,s_2}$, where $s_1,s_2>1$ and $s_1\geq s_2.$
Then there exist $T_\mu=T_\mu(\|\phi\|_{s_1,s_2},\mu)$ and a unique $u_\mu
\in C([0,T_\mu];H^{s_1,s_2}),$  satisfying the integral equation
\begin{equation}\label{integralequation}
u_{\mu}(t)=E_{\mu}(t)\phi-\int_{0}^{t}E_{\mu}(t-t')\frac{1}{2}\partial_x(u_\mu^2)(t')dt'.
\end{equation}
\end{theorem}
\begin{proof}
The proof is based on the contraction principle. Consider the complete metric
space
$$
X_{s_1,s_2}(T)=\Big\{f\in C([0,T];H^{s_1,s_2})\ | \ \|f(t)-E(t)\phi\|_{s_1,
s_2}\leq \|\phi\|_{s_1,s_2}, \forall t\in [0,T]\Big\},
$$
with the supremum norm. Define the operator
$$
Af(t)=E_{\mu}(t)\phi-\int_{0}^{t}E_{\mu}(t-t')(ff_{x})(t')dt'.
$$
The idea is to show that $A$ has a unique fixed point in $X_{s_1,s_2}(T)$ for
some $T>0$ to be chosen later. For any $f\in X_{s_1,s_2}(T)$, from
Proposition \ref{propa1}, we get
\begin{equation*}
\begin{split}
\|J_{x}^{s_{1}}E(t-t')\partial_{x}f^2\|\leq& \ c_{\mu}(1+(t-t')^{-1/2})\|J_{x}^{s_{1}-1}\partial_{x}f^2\|\\
\leq &\ c_{\mu,s_1,s_2}(1+(t-t')^{-1/2})\|\phi\|_{s_1s_2}^2.
\end{split}
\end{equation*}
Since $\frac{s_2}{s_1}\leq 1$, there exists a real number $\alpha$ satisfying
$\frac{s_2}{s_1} \leq \alpha\leq 1$. Thus,
\begin{equation*}
\begin{split}
\|J_{y}^{s_{2}}E(t-t')\partial_{x}f^2\|\leq&
c_{\alpha,\mu}(1+(t-t')^{-\alpha/2})(\|J_{x}^{s_1}f^2\|+\|J_{y}^{s_2}f^2\|)\\
\leq & \ c_{\alpha,\mu}(1+(t-t')^{-\alpha/2})\|f^2\|_{s_1,s_2}\\
\leq &  \ c_{\alpha,\mu}(1+(t-t')^{-\alpha/2})\|\phi\|_{s_1,s_2}^2,
\end{split}
\end{equation*}
where  we have used Plancherel's theorem and Young's inequality (with
$p=s_1$, $q=\frac{s_1}{s_{1} -1}$). Therefore, from the above inequalities it
transpire that
$$
\|Af(t)-E(t)\phi\|_{s_1,s_2}\leq
\left[c\|\phi\|_{s_1,s_2}\int_{0}^{t}(1+(t-t')^{-1/2}+(t-t')^{-\alpha/2})dt'\right]\|\phi\|_{s_1,s_2}.
$$
As a consequence, there exists a $T_{\mu}'=T_{\mu}'(\mu,\|\phi\|_{s_1,s_2})$
such that $A:X_{s_1,s_2}(T_{\mu}')\to X_{s_1,s_2}(T_{\mu}')$. By using
similar estimates we also show that  $A:X_{s_1,s_2}(T_\mu)\to
X_{s_1,s_2}(T_\mu)$ is a contraction. The Banach fixed point theorem gives
the desired. This completes the proof of the theorem.
\end{proof}

\begin{remark}\label{bootstrapping}
Using the integral equation \eqref{integralequation},  part a) in Proposition
\ref{propa1}, and a bootstrapping argument we can show that ${\displaystyle
u_\mu \in H^{\infty,\infty}=\bigcap_{s_1,s_2\in \R} H^{s_1,s_2}}$ for all
$t\in (0,T]$ and $\mu>0$ (see, for instance, \cite[Theorem 3.3]{Iorio}).
\end{remark}

\begin{proposition}\label{modprop}
Let $s_2>2$ and $s_1 \geq s_2$. If $u\in \mathcal{S}(\R^2)$ is real then
$$|(u,uu_{x})_{s_1,s_2}|\leq c\|u\|^{3}_{s_1, s_2}.$$
\end{proposition}
\begin{proof}
Write
\begin{equation}
\begin{split}\label{2012}
(u,uu_{x})_{s_1,s_2}=&\ (J_{x}^{s_1}u,J_{x}^{s_1}(uu_{x}))+(J_{y}^{s_2}u,J_{y}^{s_2}(uu_{x}))\\
=&\ (J_{x}^{s_1}u,[J_{x}^{s_1},u]u_x)+(J_{x}^{s_1}u,uJ_{x}^{s_1}u_{x})+(J_{y}^{s_2}u,[J_{y}^{s_2},u]u_x)\\
&+(J_{y}^{s_2}u,uJ_{y}^{s_2}u_{x}).
\end{split}
\end{equation}
Fixing $y$ in  Lemma X1 of \cite{KP}, we obtain
\begin{equation*}
\|[J_{x}^{s_1},u]u_{x}\|_{L^{2}_{x}}\leq c(\|u_{x}\|_{L^{\infty}_{x}}\|J^{s_{1}-1}u_{x}\|_{L^{2}_{x}}
+\|J_{x}^{s_1}u\|_{L^{2}_{x}}\|u_{x}\|_{L^{\infty}_{x}}).
\end{equation*}
Calculating the $L^{2}$ norm  in $y$,  using  H\"older's inequality and
\eqref{embLinfty}, we obtain
\begin{equation}\label{2010}
\|[J_{x}^{s_1},u]u_{x}\|\leq c \|u\|_{s_1,s_2}^{2}.
\end{equation}
Using similar arguments and Young's inequality, we deduce that
\begin{equation}\label{2011}
\|[J_{y}^{s_2},u]u_{x}\|\leq c \|u\|_{s_1,s_2}^{2}.
\end{equation}
Note now that integrating by parts yields
\begin{equation}\label{2013}
\begin{split}
(J_{x}^{s_1}u,uJ_{x}^{s_1}u_{x})=&\ (J_{x}^{s_1}u\partial_{x}(J_{x}^{s_1}u),u)\\
=&\ \frac{1}{2}(\partial_{x}(J_{x}^{s_1}u)^{2},u)\\
=&\ -\frac{1}{2}((J_{x}^{s_1}u)^{2},\partial_{x}u)\\
\leq& \ \|u_{x}\|_{L^{\infty}}\|J_{x}^{s_1}u\|^{2}\\
\leq & \ c\|u\|_{s_1,s_2}^{3}.
\end{split}
\end{equation}
In a similar fashion,
\begin{equation}\label{2007}
(J_{y}^{s_2}u,uJ_{y}^{s_2}u_{x})\leq c\|u\|_{s_1,s_2}^{3}.
\end{equation}
From \eqref{2012}--\eqref{2007} and the  Cauchy-Schwartz inequality, we
obtain the result.
\end{proof}

\begin{remark}\label{obs111}
Once we have proved Theorem \ref{modtheorem} and Proposition \ref{modprop},
the conclusion of Theorem \ref{105} is standard. Indeed, by using Proposition
\ref{modprop} one can show that the solution, $u_\mu\in
C([0,T_\mu];H^{s_1,s_2})$, obtained in Theorem \ref{modtheorem} can be
extended, for all $\mu>0$, to an interval $[0,T]$, where $T$ depends only on
$s_1,s_2$ and $\|\phi\|_{s_1,s_2}$ but not $\mu$. Moreover, there exists a
function $\rho\in C([0,T];\R_+)$ such
\begin{equation}\label{rho}
\|u_\mu\|^2_{s_1,s_2}\leq \rho(t), \quad \rho(0)=\|\phi\|^2_{s_1,s_2}, \quad
t\in [0,T].
\end{equation}
This in turn enable us to pass the limit in \eqref{2}, as $\mu\to0$, to
obtain a solution for \eqref{bozk} in $H^{s_1,s_2}$. The interested reader
will find all the arguments in \cite{Iorio} (see also \cite{Aniura1}, where
the author deals with a two-dimensional model). The continuous dependence
upon the data can be obtained by using the Bona-Smith approximation.
\end{remark}

\section{Local well-posedness in $H^{s}(w^2)$}\label{localweighted}

This section is devoted to prove Theorem \ref{H}. We start with the following
lemma. A similar result has also appeared in \cite{Aniura1}.

\begin{lemma}\label{lema4.1}
Let $w$ be a smooth weight with all its first, second, and third derivatives
bounded. Define
$$
w_\lambda(x,y)=w(x,y)e^{-\lambda(x^2 +y^2)},\quad (x,y)\in \R^2,\;\; \lambda
\in (0,1).
$$
Then, there exist constants $c_{j}, j=1,2,3$, independents of $\lambda$, such
that
$$\|\nabla w_\lambda\|_\infty \leq c_1,$$
$$\|D^{\alpha}w_\lambda\|_\infty\leq c_2 ,$$
$$\|D^{\beta}w_\lambda\|_\infty\leq c_3,$$
where $\alpha, \beta \in \N^2$, with $|\alpha|=2$ and $|\beta|=3.$
\end{lemma}
\begin{proof}
Let $r=\sqrt{x^2+y^2}$. From the mean-value theorem,
\begin{equation}\label{lambda}
|w(x,y)-w(0,0)|\leq r \|\nabla w\|_\infty.
\end{equation}
Thence, $|w(x,y)|\leq r \|\nabla w\|_\infty+|w(0,0)|.$  Because
$\partial_{x}w_\lambda =(w_x -2\lambda x w)e^{-\lambda r^2}$, we get
$$
|\partial_x w_\lambda|\leq \|w_x\|_\infty+\|\nabla w\|_\infty +|w(0,0)|,
$$
where  we used the inequality $r^a e^{-\lambda r^2}\leq c_a \lambda^{-a/2}$,
which is valid for all $ \lambda ,a >0 $. Also, since
$$
\partial_{x}^2 w_\lambda =(w_{xx}-4\lambda x w_x -2\lambda w +4\lambda^2 x^2 w)e^{-\lambda r^2},
$$
we deduce
$$
|\partial_{x}^2 w_\lambda|\leq c(\|w_{xx}\|_\infty +\|\nabla w\|_{\infty} +|w(0,0)|).
$$
The computations for the second-order mixed derivatives are similar. Finally,
we have
$$
\partial_{x}^{3}w_\lambda = (w_{xxx}-6\lambda x w_{xx}-6\lambda w_x
 + 12\lambda^2 x^2 w_x +12\lambda^2 x w -8\lambda^3 x^3 w)e^{-\lambda r^2}.$$
As above, there exists $c_3$, independent of $\lambda,$ such that
$\|\partial_{x}^{3}w_\lambda\|_\infty \leq c_3$. For the third-order mixed
derivatives the argument is analogue. The proof of the lemma is thus
completed.
\end{proof}

Let $w$ satisfy the hypotheses of  Lemma \ref{lema4.1}. For any $\lambda \in
(0,1)$, inequality \eqref{lambda} implies that there exists $c_{\lambda}>0$,
depending on $\lambda$, such that
$$
|w(x,y)|\leq c_{\lambda} e^{\lambda(x^2+y^2)}, \forall x,y\in \R.
$$

Now we are able to prove Theorem \ref{H}.

\begin{proof}[Proof of Theorem \ref{H}]

\underline{Existence and uniqueness:} Assume that $\phi\in H^s(w^2)$, $s>2$.
In view of Theorem \ref{120} and Remark \ref{obs111} there exists $T>0$, such
that, for all $\mu\geq0$, the unique solutions (in $H^s$) of \eqref{bozk} and
\eqref{2} are defined in the interval $[0,T]$ and satisfy
\begin{equation}\label{rhoest}
\|u_\mu(t)\|^2_{H^s}\leq \rho(t), \qquad t\in [0,T].
\end{equation}
Here, it is understood that $u_0:=u$ and $u_\mu$  are the solutions of
\eqref{bozk} and \eqref{2}, respectively. Let $M:=\sup_{t\in
[0,T]}\|u_\mu(t)\|_{H^s}$. From \eqref{rhoest}, it may be assumed that $M$
does not depend on $\mu\geq0$.\\

\noindent \underline{Persistence:} To simplify the notation, in what follows
we write, for $\mu>0$, $u_\mu=v$. Let  $w_{\lambda}$ be as in  Lemma
\ref{lema4.1}. Using Remark \ref{bootstrapping}, multiplying the equation
\eqref{2} by $w_{\lambda}^{2}v$ and integrating on $\R^2$, we obtain
\begin{equation}\label{muestL2}
\frac{1}{2}\frac{d}{dt}\|w_{\lambda}v\|^2 = (w_\lambda
v,-w_{\lambda}\mathcal{H}\partial_{x}^2 v- w_{\lambda}v_{xyy} -w_{\lambda} v
v_x + \mu w_{\lambda}\Delta v).
\end{equation}
Let us estimate the right-hand side of \eqref{muestL2}. Since $(w_\lambda v,
\mathcal{H}\partial_{x}^2 (w_\lambda v))=0$, we can write
\begin{equation}\label{teow2a}
\begin{split}
(w_\lambda v,w_\lambda \mathcal{H}\partial_{x}^2 v)=\ (w_{\lambda}v,
[w_{\lambda},\mathcal{H}]\partial_{x}^2 v)
+(w_{\lambda}v,\mathcal{H}[w_\lambda, \partial_{x}^2]v).
\end{split}
\end{equation}
By using Theorem \ref{Comu}, H\"older's inequality, and Lema \ref{lema4.1},
we obtain
$$
\|[w_\lambda,\mathcal{H}]\partial_{x}^2 v\|\leq c \|\partial_{x}^2
w_{\lambda}\|_{\infty}\|v\|\leq cM.
$$
Also, since  $\mathcal{H}$ is an isometry on $L^{2}(\R)$, Lemma \ref{lema4.1}
implies
$$
\|\mathcal{H}[w_\lambda, \partial_{x}^2]v\|=\|[w_\lambda,\partial_{x}^2]v\|=
\|\partial_{x}^2 w_{\lambda}v+2\partial_{x}w_\lambda \partial_{x}v\|\leq cM.
$$
Thus, from \eqref{teow2a},
\begin{equation}\label{teow2b}
\begin{split}
(w_\lambda v,w_\lambda \mathcal{H}\partial_{x}^2 v)\leq cM\|w_{\lambda}v\|.
\end{split}
\end{equation}
Next, we note that
\begin{equation}\label{teow2c}
\begin{split}
(w_\lambda v,w_\lambda v_{xyy})=&\ (w_\lambda v,[w_\lambda , \partial_{xyy}^3]v)+(w_\lambda v,\partial_{xyy}^3(w_\lambda v))\\
\leq & \ cM\|w_{\lambda}v\|,
\end{split}
\end{equation}
where we used that $(w_\lambda v, \partial_{xyy}^3 (w_\lambda v))=0$ and
Lemma \ref{lema4.1} to get
$$
\|[w_\lambda,\partial_{xyy}^3]v\|=\|\partial_{xyy}^3 w_\lambda v+2\partial_{xy}^2 w_\lambda \partial_y
v+ \partial_{x}w_\lambda \partial_{y}^2 v + \partial_{y}^2 w_{\lambda}
\partial_{x}v +2\partial_y w_\lambda \partial_{yx}^2 v\|\leq cM.
$$
Integrating by parts, we see that
\begin{equation}\label{teow2d}
\begin{split}
 (w_\lambda v,w_\lambda \Delta v)=&(w_\lambda v ,[w_\lambda,\Delta]v)+(w_\lambda v,\Delta (w_\lambda v))\\
=&(w_\lambda v ,[w_\lambda,\Delta]v)-\|\nabla(w_\lambda v)\|^{2}\\
\leq& \  |(w_\lambda v ,[w_\lambda,\Delta]v)|\\
\leq&  \ cM\|w_{\lambda}v\|,
\end{split}
\end{equation}
where we have used  Lemma \ref{lema4.1} to obtain
$$
\|[w_\lambda, \Delta]v\|=\|(\Delta w_\lambda)v-2\nabla w_\lambda \cdot \nabla
v\|\leq \|\Delta w_\lambda\|_\infty \|v\|+2\|\nabla w_\lambda\|_\infty
\|\nabla v\|\leq cM.
$$
Finally, we have
\begin{equation}\label{teow2e}
|(w_\lambda v,w_\lambda v v_x)|\leq \|v_x\|_\infty \|w_\lambda v\|^2 \leq M
\|w_\lambda v\|^2.
\end{equation}
Therefore, gathering together \eqref{teow2a}-\eqref{teow2e},
$$
\frac{d}{dt}\|w_\lambda v(t)\|^2 \leq c^2 M^2 +(\mu^2 +1 +cM)\|w_\lambda v(t)\|^2.
$$
By Gronwall's lemma (see e.g. \cite[page 369]{Hille}), we then deduce that
\begin{equation}\label{teow2f}
\|w_\lambda v(t)\|^2\leq \|w_\lambda \phi\|^2+tc^2M^2 + \int_{0}^{t}g_\lambda
(s)ds \quad t\in [0,T],
\end{equation}
where $g_{\lambda}(s)=(\|w_\lambda \phi\|^2+sc^2M^2)(\mu^2 +1 +cM)\exp
[s(\mu^2 +1 +cM)].$ Note that the constant $c$ in \eqref{teow2f} is
independent of $\lambda$. Thus, taking the limit, as $\lambda \to 0$, using
the monotone convergence theorem and inequality \eqref{rho}, we obtain
\begin{equation}\label{60}
\begin{split}
\|w v(t)\|^2&\leq \|w \phi\|^2+tc^2M^2 + \int_{0}^{t}g_{0}(s)ds\\
            &\leq \|w \phi\|^2+tc^2\rho(t) + \int_{0}^{t}g_{0}(s)ds\\
            &=\|w \phi\|^2+G(t,\mu)^2,\quad t\in [0,T],
\end{split}
\end{equation}
where $G$ is a continuous function on both parameters. Therefore, \eqref{60}
shows the persistence of the solution $u_\mu$, for all $\mu>0.$

Fixed $\lambda\in(0,1)$, using inequality \eqref{rhoest}, equation \eqref{2},
and Gronwall's lemma it is not difficult to prove that $\{u_{\mu}\}_{\mu> 0}$
is a Cauchy net in $L^{2}_{w_{\lambda}}=L^{2}(w_{\lambda}^{2}dxdy)$ and
$u_{\mu}\to u$ in $L^{2}_{w_{\lambda}}$, as $\mu\downarrow 0$. Therefore, if
$\varphi\in L^{2}_{w_{\lambda}}$, from \eqref{60}, we have
$$|(u,\varphi)_{L^{2}_{w_{\lambda}}}|=\lim_{\mu\to 0}|(u_{\mu},\varphi)_{L^{2}_{w_{\lambda}}}|
\leq \|\varphi\|_{L^{2}_{w_{\lambda}}}\lim_{\mu\to 0}(\|w_{\lambda}
\phi\|+G(t,\mu)).
$$
So, taking the supremum over all function $\varphi$ with
$\|\varphi\|_{L^{2}_{w_\lambda}}=1,$ in the above inequality, we find that
 \begin{eqnarray}\label{F}
\|w_{\lambda} u(t)\|\leq \|w_{\lambda} \phi\|+G(t,0),\quad t\in [0,T].
\end{eqnarray}
Next, taking the limit in \eqref{F}, as $\lambda\downarrow 0,$ and using the
monotone convergence theorem, we obtain
\begin{eqnarray}\label{D}
\|w u(t)\|\leq \|w \phi\|+G(t,0),\ \quad t\in [0,T],
\end{eqnarray}
where $G(t,0)\to 0,$ as $t\to 0.$ Inequality \eqref{D} then give us the
persistence of the solution $u$ in $L^{2}_{w}=L^{2}(w^2 dxdy).$\\

\noindent  \underline{Continuity:} We first state that $u:[0,T]\to L_{w}^2$
is weakly continuous. Indeed, for any $\varphi \in L^{2}_{w}$, define
$\varphi_{\lambda}=\varphi e^{-\lambda(x^2+y^2)}$. By the monotone
convergence theorem we have that $\varphi_{\lambda} \to \varphi$ in
$L^{2}_{w}$, as $\lambda \downarrow 0.$ Let $\epsilon>0$ be given and choose
$\lambda_{0}>0$ such that
$$
\|\varphi-\varphi_{\lambda_{0}}\|_{L^{2}_{w}}<\frac{\epsilon}{4(\|\phi\|_{L^{2}_{w}}+G(T,0))}.
$$
Fixed $t\in [0,T],$ let $\delta>0$ such that
$$
|t-s|<\delta\Rightarrow
\|u(t)-u(s)\|<\frac{\epsilon}{2\|\varphi_{\lambda_{0}}\|_{L^{2}({w^4})}}.
$$
This is possible thanks to the $H^s$ theory and the inequality
\begin{equation*}
\begin{split}
\|\varphi_{\lambda_{0}}\|_{L^{2}(w^4)}^{2}=&\ \int w^{4}|\varphi(x,y)|^{2}e^{-2\lambda_{0}(x^{2}+y^{2})}\\
\leq &\ \sup_{(x,y)\in \R^2}\big\{w^2 e^{-2\lambda_{0}(x^{2}+y^{2})}\big\}\int w^{2}|\varphi(x,y)|^{2}\\
\leq& \sup_{(x,y)\in \R^2}\big\{((x^2+y^2)\|\nabla u\|_{\infty}^2+|w(0,0)|^2) e^{-2\lambda_{0}(x^{2}+y^{2})}\big\}\int w^{2}|\varphi(x,y)|^{2}\\
\leq &\ c(w,\lambda_{0})\|\varphi\|_{L^{2}_w}<\infty.
\end{split}
\end{equation*}
Hence,  if $|t-s|<\delta$ from  \eqref{D}, we have
\begin{equation*}
\begin{split}
|(\varphi,u(t)-u(s))_{L^{2}_w}|\leq & \ |(\varphi-\varphi_{\lambda_{0}}, u(t)-u(s))_{L^{2}_w}|+|(\varphi_{\lambda_{0}},u(t)-u(s))_{L^{2}_w}|\\
\leq &  \ \|\varphi-\varphi_{\lambda_{0}}\|_{L^{2}_w}(\|u(t)\|_{L^{2}_w}+\|u(s)\|_{L^{2}_w})\\
&+|(w^2 \varphi_{\lambda_0},u(t)-u(s))|\\
\leq& \ 2\|\varphi-\varphi_{\lambda_{0}}\|_{L^{2}_w}(\|\phi\|_{L^{2}_w}+G(T,0))+\|\varphi_{\lambda_{0}}\|_{L^{2}(w^4)}\|u(t)-u(s)\|\\
<& \ \frac{\epsilon}{2}+\frac{\epsilon}{2}=\epsilon.
\end{split}
\end{equation*}
This proves our statement.

Now observe that
\begin{equation}\label{teow2g}
\begin{split}
\|u(t)-\phi\|_{L^{2}_w}^2=&\ \|u(t)\|^{2}_{L^{2}_w}+\|\phi\|^{2}_{L^{2}_w}-(\phi,u(t))_{L^{2}_w}-(\phi,u(t))_{L^{2}_w}\\
            \leq&\ G(t,0)+\|\phi\|^{2}_{L^{2}_w}+\|\phi\|^{2}_{L^{2}_w}-(\phi,u(t))_{L^{2}_w}-(\phi,u(t))_{L^{2}_w}.\\
\end{split}
\end{equation}
The weak continuity of $u$ in $L^{2}_w$ and the fact that $G(t,0)\to 0$, as
$t\to0$, then yield the   right continuity of $u$ at $t=0.$ To finish the
argument, we fixe $\tau \in (0,T)$ and use the invariance of the solution by
the translation
$$
(t,x,y)\in [0,T-\tau]\times \R^2\mapsto (t+\tau,x,y),
$$
to obtain that $u$ is  right  continuous in $[0,T)$. The left continuity at
$t=T$ is attained  in view of the change of variables
$$
(t,x,y)\in [0,T]\times \R^2\mapsto (T-t,x,y).
$$
Finally, using  the transformation
$$
(t,x,y)\mapsto (\tau-t,-x,-y),
$$
we conclude the left continuity. Thus $u$ is continuous in the interval
$[0,T]$.\\

\noindent \underline{Continuous dependence:} Let $u$ and $v$ be solutions of
\eqref{bozk} defined on the same interval $[0,T]$ such that
$u(x,y,0)=\phi(x,y)$ and $v(x,y,0)=\psi(x,y)$, with $\phi, \psi \in
H^{s}(w^2),$ $s>2$. Let  $u_\mu$ and $v_\mu$ be solutions of \eqref{2} with
$u_\mu(x,y,0)=\phi(x,y),$ $v_\mu(x,y,0)=\psi(x,y)$.  By denoting
$z=u_\mu-v_\mu,$ we see that
$$z_t +\mathcal{H}\partial_{x}^2 z+z_{xyy}+zu_x+v_\mu z_x=\mu \Delta z.$$
Multiplying the above equation by $w_{\lambda}^2 z$ and integrating on
$\R^2$, we get
\begin{equation}\label{continuousdep}
\frac{1}{2}\frac{d}{dt}\|wz\|^2+(w_\lambda z,
w_\lambda\mathcal{H}\partial_{x}^2 z+w_\lambda z_{xyy} +w_\lambda
zu_x+w_\lambda v_\mu z_x -w_\lambda\mu \Delta z)=0.
\end{equation}

Let
$\tilde{M}=\sup_{[0,T]}\{\|u_\mu(t)\|_{H^{s}(w^2)}+\|v_\mu(t)\|_{H^{s}(w^2)}\},$
then by \eqref{rhoest} and \eqref{D} $\tilde{M}$ is bounded by a constant
that is independent of $\mu$ and $\tilde{M}=O(\|\phi\|_{H^{s}(w^2)})$, as
$\phi \to \psi$, in $H^{s}(w^2)$.

By using similar computations as above, we face the inequality
$$
\frac{d}{dt}\|w_\lambda z\|^2 \leq k_1 \|w_\lambda z\|^2 +k_2
\|z\|_{L^{\infty}_{T}L^2}^2, \qquad 0\leq t\leq T,
$$
where $k_1$ and $k_2$ are constants depending only on $\tilde{M}.$ Then, by
the Gronwall lemma
\begin{equation}\label{dep5}
\|w_\lambda z\|^2 \leq (\|w_\lambda (\phi -\psi)\|+k_2 T \|z\|_{L^{\infty}_{T}L^2}^2)e^{k_1 T}.
\end{equation}
Taking the limit in \eqref{dep5}, as $\mu \downarrow 0$, we obtain
\begin{equation}\label{dep6}
\|w_\lambda (u-v)\|^2 \leq (\|w_\lambda (\phi -\psi)\|+k_2 T \|u-v\|_{L^{\infty}_{T}L^2}^2)e^{k_1 T}.
\end{equation}
Finally, as $\lambda \downarrow 0$,  \eqref{dep6} implies
\begin{equation}\label{dep7}
\|w(u-v)\|^2 \leq (\|w(\phi -\psi)\|+k_2 T \|u-v\|_{L^{\infty}_{T}L^2}^2)e^{k_1 T}.
\end{equation}
From \eqref{dep7} and the continuous dependence in $H^s(\R^2)$, we see that
$u\to v$ in $H^s(w^2)$ as $\phi \to \psi$ in $H^s(w^2).$ The proof of Theorem
\ref{H} is thus completed.
\end{proof}

\begin{remark}\label{remh2teo}
It is easily seen that the weight $w(x,y)=(1+x^2+y^2)^{\gamma/2}, \ \gamma\in
[0,1],$ satisfies the hypothesis of Theorem \ref{H}. Additional information
concerning the spaces $H^{s}(w^{2})$ can be found in \cite{Aniura1}.
\end{remark}

\section{Local  well-posedness in  $\mathcal{Z}_{s,r}$} \label{localweighted1}

In this section, we prove Theorem \ref{B1}. So, let us assume that
$\phi\in\mathcal{Z}_{s,r}$. First of all, we note that the existence of a
local solution, say $u:[0,T]\to H^s$, is obtained from Theorem \ref{120}.
Thus, we need to handle only with the space $L^2_r$. Moreover, once we obtain
the persistence property in $L^2_r$, the continuity of $u:[0,T]\to L^2_r$ and
the continuity of the map data-solution follow as in Theorem \ref{H}.\\

\noindent {\bf Part i).} The first part, that is,   the case $s>2$ and $r\in
[0,1]$ was proved in Theorem \ref{H} (see also Remark \ref{remh2teo}).
Therefore, it remains to consider $r\in(1,5/2)$.  We divide this part into
two cases.\\

\noindent {\bf Case a):} $r\in (1,2]$ and  $s\geq 2r$.  Write $r=1+\theta,$
with $\theta \in (0,1]$. Define
$$
M_{1}:=\sup_{[0,T]}(\|u\|_{H^s}+\|\langle x,y\rangle^{\theta}u\|).
$$
Since $\theta\in(0,1]$, the second term in $M_1$  is finite by the first part
of the theorem.

Let $w_N$ be as in  \eqref{defWN}. We multiply  the differential equation
\eqref{bozk} by $w_{N}^{2+2\theta}u$ and integrate on $\R^{2}$ to obtain
\begin{equation}\label{106}
\frac{1}{2}\frac{d}{dt}\|w_{N}^{1+\theta}u\|^{2}+
(w_{N}^{1+\theta}u,w_{N}^{1+\theta}\mathcal{H}\partial_{x}^{2}u+\beta
w_{N}^{1+\theta}u_{xyy}+w_{N}^{1+\theta}uu_{x})=0.
\end{equation}
Following the ideas in \cite{GermanPonce}, we can write
\begin{equation*}
\begin{split}
w_{N}^{1+\theta}\mathcal{H}\partial_{x}^{2}u=&\ [w_{N}^{1+\theta};\mathcal{H}]\partial_{x}^{2}u+\mathcal{H}(w_{N}^{1+\theta}\partial_{x}^{2}u)\\
=&\ A_1+\mathcal{H}\partial_{x}^{2}(w_{N}^{1+\theta}u)-
2\mathcal{H}(\partial_{x}w_{N}^{1+\theta}\partial_{x}u)-\mathcal{H}\partial_{x}^{2}w_{N}^{1+\theta}u\\
=&\ A_1 +A_2 +A_3 +A_4.
\end{split}
\end{equation*}
In view of Theorem \ref{Comu}, we have
\begin{equation}\label{A1estimate}
\begin{split}
\|A_1\|&=\|\|[w_{N}^{1+\theta};\mathcal{H}]\partial_{x}^{2}u\|_{L^{2}_{x}}\|_{L^{2}_{y}}\\
&\leq
c\|\|\partial_{x}^{2}w_{N}^{1+\theta}\|_{L^{\infty}_{x}}\|u\|_{L^{2}_{x}}\|_{L^{2}_{y}}\leq
\;c\|\partial_{x}^{2}w_{N}^{1+\theta}\|_{L^{\infty}_{xy}}\|u\| \leq  \;
cM_{1}.
\end{split}
\end{equation}
An application of Lemma \ref{inter}, with $a=1+2\theta, \
\alpha=\frac{1}{1+2\theta}$, and $b=1/2+\theta$ yields
\begin{equation}\label{A3estimate}
\|A_3\|\leq 2(1+\theta)\Big(\|\partial_{x}(w_{N}^{\theta}u)\|+\|u\|\Big)\leq
c(\|J^1(w_{N}^{\theta}u)\|+\|u\|)\leq c(\|w_{N}^{1+\theta}u\|+M_1),
\end{equation}
and, similarly,
\begin{equation}\label{A4estimate}
\|A_4\|\leq cM_{1}.
\end{equation}
Furthermore, inserting $A_2$ in \eqref{106} we find that its contribution is
null.  The constant $c$ that appears here and in the rest of the proof of the
theorem will always be independent of $N$. From  Lemma \ref{inter}, with
$a=2+2\theta, \ \alpha=\frac{1}{2+2\theta}$, and \ $b=1+\theta$, we obtain
\begin{equation}\label{teoZa}
\|J^1(w_{N}^{1/2+\theta}u)\|\leq c(\|w_{N}^{1+\theta}u\|+\|J^{2+2\theta}u\|+M_1).
\end{equation}
Another application of Lemma \ref{inter}, with $a=2+2\theta, \
\alpha=\frac{2}{2+2\theta}$, and $b=1+\theta$ implies
\begin{equation}\label{teoZb}
\|J^{2}(w_{N}^{\theta}u)\|\leq c(\|w_{N}^{1+\theta}u\|+\|J^{2+2\theta}u\|+M_1).
\end{equation}

Using integration by parts, the  inequality
$|\partial_{x}w_{N}^{2+2\theta}|\leq cw_{N}^{1+2\theta}$, \eqref{teoZa} and
\eqref{teoZb}, we obtain
\begin{equation}\label{teoZc}
\begin{split}
\int w_{N}^{2+2\theta}u\partial_{x}\partial_{y}^{2}u=&\ \frac{1}{2}\int(-2\partial_{y}w_{N}^{2+2\theta}u\partial_{x}\partial_{y}u+
\partial_{x}w_{N}^{2+2\theta}(\partial_{y}u)^2 )\\
\leq& \
\|w_{N}^{1+\theta}u\|\|w_{N}^{\theta}\partial_{x}\partial_{y}u\|+\|w_{N}^{1/2+\theta}\partial_{y}u\|^2\\
\leq & \ c(\|J^{2}(w_{N}^{\theta}u)\|^2 +\|J(w_{N}^{1/2+\theta}u)\|^2
+\|w_{N}^{1+\theta}u\|^2+M_{1}^{2})\\ \leq &\
c(\|w_{N}^{1+\theta}u\|^2+M_{1}^2).
\end{split}
\end{equation}

Finally, since $s>2$, Sobolev's embedding gives
\begin{equation}\label{teoZd}
(w_{N}^{1+\theta}u,w_{N}^{1+\theta}uu_{x})\leq M_{1}\|w_{N}^{1+\theta}u\|^2.
\end{equation}

From \eqref{106}, H\"older's and the  above inequalities,  we find that
$$
\frac{d}{dt}\|w_{N}^{1+\theta}u\|^{2}\leq c(1+\|w_{N}^{1+\theta}u\|^{2}).
$$

So, by the Gronwall lemma, we get
$$
\|w_{N}^{1+\theta}u\|^{2}\leq \|w_{N}^{1+\theta}\phi\|^{2}+tc+c\int_{0}^{t}e^{ct'}(\|w_{N}^{1+\theta}\phi\|^{2}+t'c)dt'.
$$

The monotone convergence theorem then yields
\begin{equation}\label{107}
\|\langle x,y \rangle^{1+\theta}u\|^{2}\leq \|\langle x,y
\rangle^{1+\theta}\phi\|^{2}+g(t),
\end{equation}
where $g(t)\to 0,$  as  $t \downarrow 0.$ This  proves the persistence
property in $L^2_r$. As we already pointed out, the continuity follows as in
Theorem \ref{H}.\\

\noindent {\bf Case b):} $r\in (2,5/2)$ and $s\geq2r$. Let $r=2+\theta$, with
$\theta \in (0,1/2)$. Define
$$
M_{2}=\displaystyle\sup_{[0,T]}\{\|\langle x,y\rangle^{2}u\|+\|u\|_{H^{s}}\}.
$$
We now multiply the differential equation  \eqref{bozk} by
$x^{2}w_{N}^{2+2\theta}u$ and integrate on $\R^2$ to obtain
\begin{equation}\label{300}
\begin{split}
\frac{1}{2}\frac{d}{dt}&\|xw_{N}^{1+\theta}u\|^{2}\\
&\leq |(xw_{N}^{1+\theta}u, xw_{N}^{1+\theta}\mathcal{H}\partial_{x}^{2}u)+
(xw_{N}^{1+\theta}u, x
w_{N}^{1+\theta}u_{xyy})+(xw_{N}^{1+\theta}u,xw_{N}^{1+\theta}uu_{x})|.
\end{split}
\end{equation}
Let us control the first term on the right-hand side of \eqref{300}. Since
$\partial_{x}^{2}(xu)=2\partial_{x}u+x\partial_{x}^{2}u$ and
$\mathcal{H}(x\partial_xu)=x\mathcal{H}(\partial_xu)$,  we  can write
$$
x\mathcal{H}\partial_{x}^{2}u=\mathcal{H}\partial_{x}^{2}(xu)-2\mathcal{H}\partial_{x}u=B_{1}+B_{2}.
$$
By definition of $w_N$, we deduce the inequality
\begin{equation}\label{wNjap}
w_N^{1+\theta}(x,y)\leq \langle x,y\rangle^{1+\theta}\leq (1+|x|+|y|)\langle
x,y\rangle^{\theta}.
\end{equation}
Using \eqref{wNjap},  Theorems \ref{boundhilbert} and  \ref{indc}, Remark
\ref{condap},  and  the identity $\widehat{\partial_{x}u}(0,\eta,t)=0$ we get
\begin{equation*}
\begin{split}
\|w_{N}^{1+\theta}B_{2}\|\leq & \ c \|w_{N}^{1+\theta}\mathcal{H}\partial_{x}u\|\\ \nonumber
\leq& \ c \|w_{N}^{\theta}\mathcal{H}\partial_{x}u\|+\|xw_{N}^{\theta}\mathcal{H}\partial_{x}u\|+\|yw_{N}^{\theta}\mathcal{H}\partial_{x}u\|\\
\leq & \ c\|\langle x,y\rangle^{\theta}\mathcal{H}\partial_{x}u\|+\|x\langle x,y\rangle^{\theta}\mathcal{H}\partial_{x}u\|+\|y\langle x,y\rangle^{\theta}\mathcal{H}\partial_{x}u\|\\
\leq & \ c\|\langle x,y\rangle^{\theta}\mathcal{H}\partial_{x}u\|+\|\langle x,y\rangle^{\theta}\mathcal{H}(x\partial_{x}u)\|+\|\langle x,y\rangle^{\theta}\mathcal{H}(y\partial_{x}u)\|\\
\leq & \ c(\|\mathcal{H}\partial_{x}u\|+\||x|^{\theta}\mathcal{H}\partial_{x}u\|+\||y|^{\theta}\mathcal{H}\partial_{x}u\|+
\|\mathcal{H}(x\partial_{x}u)\|+\||x|^{\theta}\mathcal{H}(x\partial_{x}u)\|\\
&+\||y|^{\theta}\mathcal{H}(x\partial_{x}u)\|+\|\mathcal{H}(y\partial_{x}u)\|+\||x|^{\theta}\mathcal{H}(y\partial_{x}u)\|+
\||y|^{\theta}\mathcal{H}(y\partial_{x}u)\|\\
\leq& \ c(\|\partial_{x}u\|+c^{*}\||x|^{\theta}\partial_{x}u\|+\||y|^{\theta}\partial_{x}u\|+\|x\partial_{x}u\|+c^{*}\||x|^{\theta}x\partial_{x}u\|\\
&+\||y|^{\theta}x\partial_{x}u\|+\|y\partial_{x}u\|+c^*\||x|^{\theta}y\partial_{x}u\|+\||y|^{\theta}y\partial_{x}u\|)\\
\leq& \ c\|\langle x,y\rangle^{1+\theta}\partial_{x}u\|)\\
=&\ C.
\end{split}
\end{equation*}
From Lemma \ref{inter}, it follows that
\begin{eqnarray*}
C\leq c(\|J(\langle x,y \rangle^{1+\theta}u)\|+M_2)\leq c(\|\langle x,y
\rangle^{3/2+\theta}u\|+\|J^{3+2\theta}u\|+M_2)\leq cM_2.
\end{eqnarray*}

To estimate the term with $B_1$, note that
\begin{equation*}\label{101}
\begin{split}
w_{N}^{1+\theta}\mathcal{H}\partial_{x}^{2}(xu)=&\
[w_{N}^{1+\theta},\mathcal{H}]\partial_{x}^{2}(xu)+\mathcal{H}(w_{N}^{1+\theta}\partial_{x}^{2}(xu))\\
=& \ D_1 +\mathcal{H}(\partial_{x}^{2}(w_{N}^{1+\theta}xu))-2\mathcal{H}(\partial_{x}w_{N}^{1+\theta}\partial_{x}(xu))-\mathcal{H}(\partial_{x}^{2}w_{N}^{1+\theta}xu)\\
=& \ D_1 +D_2+D_3+D_4.
\end{split}
\end{equation*}
Inserting $D_2$ in \eqref{300} one has that its contribution is null. Furthermore, using similar arguments as above,
\begin{equation*}
\|D_1\|\leq cM_2, \qquad  \ \|D_4\|\leq cM_2.
\end{equation*}
To control $D_3$, we use that $|\partial_xw_N|\leq1$, to obtain
\begin{equation*}
\|D_3\|=2\|\mathcal{H}(\partial_{x}w_{N}^{1+\theta}\partial_{x}(xu))\|\leq c(\|w_{N}^{\theta}u\|+\|xw_{N}^{\theta}\partial_{x}u\|)\leq cM_2.
\end{equation*}

Next, we will control the middle term on the right-hand side of \eqref{300}.
The estimates  $|x\partial_{x}w_{N}|\leq 3w_{N}$ and
$|x\partial_{y}^{2}w_{N}|\leq 1$  give the inequalities
\begin{equation}\label{137}
|\partial_{x}(x^{2}w_{N}^{2+2\theta})|\leq c|xw_{N}^{2+2\theta}|, \ \quad\  |\partial_{y}^{2}(x^{2}w_{N}^{2+2\theta})|\leq c|xw_{N}^{2+2\theta}|.
\end{equation}
Also,  Lemma \ref{inter} leads to
\begin{equation}\label{desigualdade}
\|J^{2}(w_{N}^{1+\theta}u)\|\leq c(\|w_{N}^{2+\theta}u\|+\|J^{4+2\theta}u\|).
\end{equation}
Using integrating by parts, \eqref{137} and \eqref{desigualdade} (to estimate the term with second-order derivatives), we obtain
\begin{equation*}
\begin{split}
\int x^{2} w_{N}^{2+2\theta}u\partial_{x}\partial_{y}^{2}u=&\ -\frac{1}{2}\int(\partial_{y}^{2}(x^{2}w_{N}^{2+2\theta})u\partial_{x}u+\partial_{x}(x^{2}w_{N}^{2+2\theta})\partial_{y}^{2}u u\\
&+2\partial_{y}(x^{2}w_{N}^{2+2\theta})\partial_{x}\partial_{y}u u) \\
\leq& \ \|w_{N}^{1+\theta}\partial_{x}u\|\|xw_{N}^{1+\theta}u\|+\|xw_{N}^{1+\theta}u\|^{2}+\|w_{N}^{1+\theta}\partial_{y}^{2}u\|\|xw_{N}^{1+\theta}u\|+\\
&+\|w_{N}^{1+\theta}\partial_{x}\partial_{y}u\|\|xw_{N}^{1+\theta}u\|\\
\leq & \ c(M_{2}^{2}+\|xw_{N}^{1+\theta}u\|^2 +\|w_{N}^{1+\theta}\partial_{y}^{2}u\|^2+\|w_{N}^{1+\theta}\partial_{x}\partial_{y}u\|^2)\\
\leq & \ c(M_{2}^{2}+\|xw_{N}^{1+\theta}u\|^2 +\|w_{N}^{2+\theta}u\|^2+\|J^{4+2\theta}u\|^2)\\
\leq& \ c(M_{2}^{2}+\|xw_{N}^{1+\theta}u\|^2+\|yw_{N}^{1+\theta}u\|^2).
\end{split}
\end{equation*}

\noindent Finally, the last term on the right-hand side of \eqref{300} is controlled as
\begin{equation*}
(xw_{N}^{1+\theta}u,xw_{N}^{1+\theta}uu_{x})\leq M_{2}\|xw_{N}^{1+\theta}u\|^{2}.
\end{equation*}

\noindent The H\"older inequality applied to \eqref{300},  together with the
above estimates yield
\begin{equation}\label{8}
\frac{d}{dt}\|xw_{N}^{1+\theta}u\|^{2}\leq c(M_{2}^{2}+\|xw_{N}^{1+\theta}u\|^{2}+\|yw_{N}^{1+\theta}u\|^{2}).
\end{equation}
A similar computation with $y$ instead of $x$ gives
\begin{equation}\label{140}
\frac{d}{dt}\|yw_{N}^{1+\theta}u\|^{2}\leq c(M_{2}^{2}+\|xw_{N}^{1+\theta}u\|^{2}+\|yw_{N}^{1+\theta}u\|^{2}).
\end{equation}

From \eqref{8}, \eqref{140}, Gronwall's inequality,  and the monotone
convergence theorem we are able to establish the persistence property. This
proves Case b).\\

\noindent {\bf Part ii).} We also split this case in two other ones.

\noindent {\bf Case a):} $r\in [5/2,3)$ and $s\geq2r$.  Let $r=2+\theta,$
with $\theta\in [1/2,1),$ $s\geq 2r$. Let
$$
M_3=\sup_{[0,T]}\{\|\langle x,y \rangle^{3/2+\theta}u\|+\|u\|_{H^{s}}\}.
$$
We multiply the differential equation \eqref{bozk} by $x^{4}w_{N}^{2\theta}u$
and integrate over $\R^2$ to obtain
\begin{equation}\label{600}
\frac{1}{2}\frac{d}{dt}\|x^{2}w_{N}^{\theta}u\|^{2}+(x^{2}w_{N}^{\theta}u,x^{2}w_{N}^{\theta}\mathcal{H}\partial_{x}^{2}u+
\beta x^{2}w_{N}^{\theta}u_{xyy}+x^{2}w_{N}^{\theta}uu_{x})=0.
\end{equation}
From the equality
$$
x^{2}\mathcal{H}\partial_{x}^{2}u=\mathcal{H}\partial_{x}^{2}(x^{2}u)-4\mathcal{H}\partial_{x}(xu)+2\mathcal{H}u,
$$
we have
\begin{equation*}
\begin{split}
w_{N}^{\theta}x^{2}\mathcal{H}\partial_{x}^{2}u=&\
 w_{N}^{\theta}\mathcal{H}\partial_{x}^{2}(x^{2}u)+4w_{N}^{\theta}\partial_{x}(xu)-2w_{N}^{\theta}\mathcal{H}u\\
                                               =&\ Q_1 +Q_2 +Q_3,
\end{split}
\end{equation*}
Since $\phi\in\dot{\mathcal{Z}}_{s,r}$, we deduce that
$\hat{\phi}(0,\eta)=\hat{u}(t,0,\eta)=0,$ for all $t\in [0,T],$ (see
\eqref{fourieru}), which implies that $\mathcal{H}(xu)=x\mathcal{H}u.$
Therefore, the boundedness of $\mathcal{H}$ in $L^2_x$, gives
\begin{equation*}
\begin{split}
\|Q_{3}\|=& \ 2\|w_{N}^{\theta}\mathcal{H}u\|\\
         \leq& \ c(\|\mathcal{H}u\|+\|x\mathcal{H}u\|+\|y\mathcal{H}u\|)\\
         =&\ c(\|u\|+\|\mathcal{H}(xu)\|+\|\mathcal{H}(yu)\|)\\
         \leq& \ cM_{3}.
\end{split}
\end{equation*}
Note that
\begin{equation*}
\begin{split}
Q_{2}=&\ w_{N}^{\theta}\mathcal{H}\partial_{x}(xu)\\
     =&\ [w_{N}^{\theta};\mathcal{H}]\partial_{x}(xu)+\mathcal{H}(w_{N}^{\theta}\partial_{x}(xu))\\
=&\ Q_{2}^{1}+Q_{2}^{2}.
\end{split}
\end{equation*}
A simple analysis reveals that
$$
\|Q_{2}^{1}\|\leq \|\partial_{x}w_{N}^{\theta}\|_{\infty}\|xu\|\leq cM_{3}.
$$
Moreover, Lemma \ref{inter} yields
\begin{equation*}
\begin{split}
\|Q_{2}^{2}\|\leq& \ \|w_{N}^{\theta}u\|+\|xw_{N}^{\theta}\partial_{x}u\|\\
\leq& \ \|\langle x,y \rangle^{\theta}u\|+\|\langle x,y \rangle^{\theta+1}\partial_{x}u\|\\
\leq& \ c(\|\langle x,y \rangle u\|+\|\langle x, y \rangle^{3/2+\theta}u\|+\|J^{3+2\theta}\|)\\
\leq& \ cM_{3}.
\end{split}
\end{equation*}
For $Q_1$  we can write
\begin{equation*}
\begin{split}
Q_1=&\ w_{N}^{\theta}\mathcal{H}\partial_{x}^{2}(x^{2}u)\\
   =&\ [w_{N}^{\theta};\mathcal{H}]\partial_{x}^{2}(x^{2}u)+\mathcal{H}(w_{N}^{\theta}\partial_{x}^{2}(x^{2}u))\\
=&\ V_{1}+\mathcal{H}\partial_{x}^{2}(x^{2}w_{N}^{\theta}u)-2\mathcal{H}(\partial_{x}w_{N}^{\theta}\partial_{x}(x^{2}u))
-\mathcal{H}(\partial_{x}^{2}w_{N}^{\theta}x^{2}u)\\
=& \ V_1 +V_2 +V_3 +V_4.
\end{split}
\end{equation*}
Inserting $V_2$ in \eqref{600} one has that its contribution is null. By
Theorem \ref{Comu}
\begin{equation*}
\|V_1\|\leq cM_3,\qquad\|V_4\|\leq c M_3.
\end{equation*}
In view of Lemma \ref{inter},
\begin{equation*}\nonumber
\begin{split}
\|V_3\|\leq& \
c(\|x^{2}\partial_{x}w_{N}^{\theta}\partial_{x}u\|+2\|x\partial_{x}w_{N}^{\theta}u\|)\\
\leq & \ c(\|xw_{N}^{\theta}\partial_{x}u\|+\|w_{N}^{\theta}u\|)\\
\leq & \ cM_{3}.
\end{split}
\end{equation*}
Also,
\begin{equation}
(x^{4}w_{N}^{2\theta}u,uu_{x})\leq \|u_{x}\|_{\infty}\|x^{2}w_{N}^{\theta}u\|^{2}.
\end{equation}
Similar computations  as the previous ones leads us to the inequalities
\begin{equation}\label{144}
 |\partial_{x}(x^{4}w_{N}^{2\theta})|, |\partial_{y}(x^{4}w_{N}^{2\theta})|, |\partial_{y}^{2}(x^{4}w_{N}^{2\theta})|\leq c|x^3 w_{N}^{2\theta}|.
\end{equation}
Therefore, integrating by parts, using \eqref{144} and  Lemma \ref{inter}, we
obtain
\begin{equation*}
\begin{split}
\int x^{4} w_{N}^{2\theta}u\partial_{x}\partial_{y}^{2}u=&\
-\frac{1}{2}\int\Big(\frac{1}{2}\partial_{y}^{2}(x^{4}w_{N}^{2\theta})u\partial_{x}u-\partial_{x}(x^{4}w_{N}^{2\theta})\partial_{y}^{2}u u\\
&-2\partial_{y}(x^{4}w_{N}^{2\theta})\partial_{x}\partial_{y}u u\Big) \\
 \leq
&\int |\partial_{y}^{2}(x^{4}w_{N}^{2\theta})u\partial_{x}u|+\int |x^3
w_{N}^{2\theta}\partial_{y}^{2}u u|+\int |x^3
w_{N}^{2\theta}\partial_{x}\partial_{y}u u|\\
\leq& \
\|x^{2}w_{N}^{\theta}u\|\|xw_{N}^{\theta}\partial_{x}u\|+\|xw_{N}^{\theta}\partial_{y}^{2}u\|\|x^2
w_{N}^{\theta}u\|\\\nonumber &+\|x
w_{N}^{\theta}\partial_{x}\partial_{y}u\|\|x^2 w_{N}^{\theta}u\|\\
\leq & c\Big( M_{3}^{2}+\|x^{2}w_{N}^{\theta}u\|^{2}+\|xw_{N}^{\theta}\partial_{y}^{2}u\|^{2}+\|x w_{N}^{\theta}\partial_{x}\partial_{y}u\|^{2}\Big)\\
=&  c\Big( M_{3}^2+\|x^2 w_{N}^{\theta}u\|^{2}+ F_{1} +F_{2}\Big).
\end{split}
\end{equation*}
Putting of  the above estimates together, we obtain
\begin{equation}\label{109}
\frac{d}{dt}\|x^{2}w_{N}^{\theta}u\|^2\leq
c(M_{3}^2+\|x^{2}w_{N}^{\theta}u\|^2
+\|xw_{N}^{\theta}\partial_{y}^{2}u\|^2+\|x
w_{N}^{\theta}\partial_{x}\partial_{y}u\|^2).
\end{equation}

Note that, in view of the terms $F_1$ and $F_2$, \eqref{109} is not good
enough to our purpose. So, in what follows, we will control $F_1$ and $F_2$.
The idea is to obtain an estimate of the form
\begin{equation}\label{GestteZ}
\frac{d}{dt}G\leq cG,
\end{equation}
where $G$ is a sum of terms which include $F_1$ and $F_2$.

Taking the derivative  with respect to $y$ twice in \eqref{bozk} and
multiplying it by $x^2 w_{N}^{2\theta}\partial_{y}^{2}u$, we have
\begin{equation}\label{108}
\begin{split}
\frac{1}{2}\frac{d}{dt}\|xw_{N}^{\theta}\partial_{y}^{2}u\|^{2}+&(xw_{N}^{\theta}\partial_{y}^{2}u,xw_{N}^{\theta}\mathcal{H}\partial_{x}^{2}(\partial_{y}^{2}u))\\
&+(x^2 w_{N}^{2\theta}\partial_{y}^{2}u,
\partial_{x}\partial_{y}^{2}\partial_{y}^{2}u) +(x^2
w_{N}^{2\theta}\partial_{y}^{2}u, \partial_{y}^{2}(uu_{x}))=0.
\end{split}
\end{equation}
From the identity
$$
x\mathcal{H}\partial_{x}^{2}(\partial_{y}^{2}u)=\mathcal{H}(\partial_{x}^{2}(x\partial_{y}^{2}u))-2\mathcal{H}\partial_{x}\partial_{y}^2u,
$$
we deduce that
$$
w_{N}^{\theta}x\mathcal{H}\partial_{x}^{2}(\partial_{y}^{2}u)=
w_{N}^{\theta}\mathcal{H}(\partial_{x}^{2}(x\partial_{y}^{2}u))-2w_{N}^{\theta}\mathcal{H}\partial_{x}\partial_{y}^{2}u=E_{1}+E_{2}.
$$
Write
\begin{equation*}
\begin{split}
E_{2}=&\ -2w_{N}^{\theta}\mathcal{H}\partial_{x}\partial_{y}^{2}u\\
=&\ [w_{N}^{\theta};\mathcal{H}]\partial_{x}\partial_{y}^{2}u+\mathcal{H}(w_{N}^{\theta}\partial_{x}\partial_{y}^{2}u)\\
=&\ E_{2}^{1}+E_{2}^{2}.
\end{split}
\end{equation*}
Theorem \ref{Comu} implies
$$
\|E_{2}^{1}\|\leq
\|\partial_{x}w_{N}^{\theta}\|_{\infty}\|\partial_{y}^{2}u\|\leq cM_{3}.
$$
Lemma \ref{inter}, with $a=3+2\theta,$ $\alpha=\frac{3}{3+2\theta}$ and
$b=3/2+\theta$ gives
$$\|E_{2}^{2}\|\leq \|w_{N}^{\theta}\partial_{x}\partial_{y}^{2}u\|\leq \|\langle x, y \rangle^{\theta}\partial_{x}\partial_{y}^{2}u\|\leq M_3+2(\|\langle x,y \rangle^{3/2+\theta}u\|+\|J^{3+2\theta}u\|)\leq cM_{3}.$$
Next, we write
\begin{equation*}
\begin{split}
E_1=&\ [w_{N}^{\theta};\mathcal{H}]\partial_{x}^{2}(x\partial_{y}^{2}u)+\mathcal{H}(w_{N}^{\theta}\partial_{x}^{2}(x\partial_{y}^{2}u))\\
=&\ k_{1}+\mathcal{H}\partial_{x}^{2}(xw_{N}^{\theta}\partial_{y}^{2}u)-2\mathcal{H}\partial_{x}w_{N}^{\theta}\partial_{x}(x\partial_{y}^{2}u)+
\mathcal{H}(\partial_{x}^{2}w_{N}^{\theta}x\partial_{y}^{2}u)\\
=&\ k_1 +k_2+k_3+k_4.
\end{split}
\end{equation*}
It is easily seen that
$$
\|k_1\|\leq
\|\partial_{x}^{2}w_{N}^{\theta}\|_{\infty}\|x\partial_{y}^{2}u\|\leq cM_3,
\quad\|k_4\|\leq cM_3
$$
Moreover, inserting $k_2$ in \eqref{108} we can see that its contribution is
null. In addition,
\begin{equation*}
\begin{split}
k_3\leq & \ \|\mathcal{H}\partial_{x}w_{N}^{\theta}
\partial_{y}^{2}u\|+\|\mathcal{H}\partial_{x}w_{N}^{\theta}x\partial_{x}\partial_{y}^{2}u\|\\
\leq &\
\|w_{N}^{\theta-1}\partial_{y}^{2}u\|+\|w_{N}^{\theta}\partial_{x}\partial_{y}^{2}u\|\\
\leq & \ cM_3.
\end{split}
\end{equation*}
An application of Lemma \ref{inter}, with $a=4+2\theta,$
$\alpha=\frac{2}{2+\theta}$ and $b=2+\theta,$ yields
\begin{equation}\label{200}
\|w_{N}^{\theta}\partial_{y}^{2}\partial_{y}^{2}u\|\leq
M_{3}+\|w_{N}^{2+\theta}u\|+ \|J^{4+2\theta}u\|\leq M_3 +
\|x^{2}w_{N}^{\theta}u\|+\|y^{2}w_{N}^{\theta}u\|.
\end{equation}
Similarly,
\begin{equation}\label{202}
\|w_{N}^{\theta}\partial_{x}\partial_{y}\partial_{y}^{2}u\|\leq M_3 + \|x^{2}w_{N}^{\theta}u\|+\|y^{2}w_{N}^{\theta}u\|.
\end{equation}
Taking derivatives, it is easy to see that
\begin{equation}\label{201}
|\partial_{x}(x^2 w_{N}^{2\theta})|,|\partial_{y}(x^2 w_{N}^{2\theta})|,|\partial_{y}^{2}(x^2 w_{N}^{2\theta})|\leq c|xw_{N}^{2\theta}|.
\end{equation}
Using \eqref{200}, \eqref{202} and \eqref{201}, we get
\begin{equation*}
\begin{split}
\int x^2 w_{N}^{2\theta}\partial_{y}^{2}u
\partial_{x}\partial_{y}^{2}\partial_{y}^{2}u=& -\frac{1}{2}\int
(\partial_{y}^{2}(x^2
w_{N}^{2\theta})\partial_{y}^{2}u\partial_{x}\partial_{y}^{2}u -
\partial_{x}(x^2 w_{N}^{2\theta})\partial_{y}^{2}\partial_{y}^{2}u
\partial_{y}^{2}u\\
&-2\partial_{y}(x^2 w_{N}^{2\theta})\partial_{x}\partial_{y}\partial_{y}^{2}u\partial_{y}^{2}u)dxdy\\
\leq & c\Big( M_{3}^{2}+
\|w_{N}^{\theta}\partial_{y}^{2}\partial_{y}^{2}u\|\|xw_{N}^{\theta}\partial_{y}^{2}u\|+
\|w_{N}^{\theta}\partial_{x}\partial_{y}\partial_{y}^{2}u\|\|xw_{N}^{\theta}\partial_{y}^{2}u\|\Big)\\
\leq& c\Big(
M_{3}^{2}+\|x^{2}w_{N}^{\theta}u\|^{2}+\|y^{2}w_{N}^{\theta}u\|^2+\|xw_{N}^{\theta}\partial_{y}^{2}u\|^2\Big).
\end{split}
\end{equation*}
Finally,
\begin{equation*}
\begin{split}
(x^2 w_{N}^{2\theta}\partial_{y}^{2}u, \partial_{y}^{2}(uu_x))=&\ (x^2 w_{N}^{2\theta}\partial_{y}^{2}u, \partial_{y}^{2}u u_{x}+2u_{y}u_{xy}+uu_{xyy})\\
\leq& \ \|xw_{N}^{\theta}\partial_{y}^{2}u\|^{2}\|u_{x}\|_{\infty}+\|xw_{N}^{\theta}\partial_{y}^{2}u\|(\|xw_{N}^{\theta}\partial_{y}u\|\|u_{xy}\|_{\infty}\\
&+\|xw_{N}^{\theta}u\|\|u_{xyy}\|_{\infty})\\
\leq & \ (M_{3}+1)\|xw_{N}^{\theta}\partial_{y}^{2}u\|^{2}+M_3.
\end{split}
\end{equation*}
Collecting all the above estimates, we deduce the inequality
\begin{eqnarray}\label{estF1F2e}
\frac{d}{dt}\|xw_{N}^{\theta}\partial_{y}^{2}u\|^2\leq c(M_{3}^2+\|x^{2}w_{N}^{\theta}u\|^2
+\|y^{2}w_{N}^{\theta}u\|^2+\|xw_{N}^{\theta}\partial_{y}^{2}u\|^2).
\end{eqnarray}
Analogously we can obtain an inequality  involving $F_2$.

Multiplying the equation \eqref{bozk} by $y^4 w_{N}^{2\theta}u,$ we can
obtain an estimate similar to \eqref{109}, where another two terms as those
for $F_1$ and $F_2$ appear (but now with a multiplying  factor of $y$ instead
of $x$). Thus, we can proceed as above.

As a final step, by  writing
$$
g_{1}=\|x^2w_{N}^{\theta} u\|,\quad
g_{2}=\|xw_{N}^{\theta}\partial_{y}^{2}u\|, \quad g_{3}=\|x
w_{N}^{\theta}\partial_{x}\partial_{y}u\|,
$$
and
$$
g_{4}=\|y^2w_{N}^{\theta} u\|,\quad
g_{5}=\|yw_{N}^{\theta}\partial_{y}^{2}u\|, \quad g_{6}=\|y
w_{N}^{\theta}\partial_{x}\partial_{y}u\|,
$$
we deduce the following system of inequalities
$$
\frac{d}{dt}g_{j}^{2}\leq c\sum_{i=1}^{6}g_{i}^{2}, \qquad j=1,\ldots,6.
$$
By defining  $G=\sum_{i=1}^{6}g_{i}^{2},$  we then get the desired estimate
\eqref{GestteZ}. The rest of the proof runs as in Case a).\\

\noindent {\bf Case b).} $r\in [3,7/2)$ and $s\geq 2r$. Write  $r=2+\theta,$
where $\theta\in [1,3/2)$. Define
 $$
 M_{4}=\sup_{[0,T]}\{\|\langle x, y
\rangle^{3/2+\theta}u\|+\|u\|_{H^{s}}\}.
$$
Here, the estimates are similar to those ones in Case a), except for the
terms
$$
\tilde{Q}_{3}= \ 2w_{N}^{\theta}\mathcal{H}u,
$$
$$
\tilde{Q}_{2}=\ w_{N}^{\theta}\mathcal{H}\partial_{x}(xu),
$$
and
$$
\tilde{E}_{2}=\ -2w_{N}^{\theta}\mathcal{H}\partial_{x}\partial_{y}^{2}u,
$$
which can be estimated  using Theorems \ref{boundhilbert} and \ref{indc} and
Remark \ref{condap}. Indeed,
\begin{equation*}
\begin{split}
\|\tilde{Q}_{3}\|\leq &\ 2\|\langle x, y \rangle^{\theta} \mathcal{H}u\|\\
\leq& \ c(\|\langle x, y \rangle^{\theta-1} \mathcal{H}u\|+\|x\langle x, y \rangle^{\theta-1}\mathcal{H} u\|
+\|y\langle x, y \rangle^{\theta-1}\mathcal{H} u\|)\\
\leq& \ c(\|\mathcal{H}u\|+\|x\mathcal{H}u\|+\||x|^{\theta-1}\mathcal{H}u\|+\||y|^{\theta-1}\mathcal{H}u\|+
\||x|^{\theta-1}\mathcal{H}u\|+\\
&+\|x|y|^{\theta-1}\mathcal{H}u\|+\|y\mathcal{H}u\|+
\|y|x|^{\theta-1}\mathcal{H}u\|+\|y|y|^{\theta-1}\mathcal{H}u\|)\\
=& \ c(\|\mathcal{H}u\|+\|\mathcal{H}(xu)\|+\||x|^{\theta-1}\mathcal{H}u\|+\||y|^{\theta-1}\mathcal{H}u\|+
\||x|^{\theta-1}\mathcal{H}(xu)\|+\\
&+\||y|^{\theta-1}\mathcal{H}(xu)\|+\|\mathcal{H}(yu)\|+
\||x|^{\theta-1}\mathcal{H}(yu)\|+\||y|^{\theta-1}\mathcal{H}(yu)\|)\\
\leq& \ c(\|u\|+\|xu\|+c^{*}\||x|^{\theta-1}u\|+\|\mathcal{H}(|y|^{\theta-1}u)\|+
c^{*}\||x|^{\theta-1}xu\|+\\
&+\|\mathcal{H}(|y|^{\theta-1}xu)\|+\|yu\|+
c^{*}\||x|^{\theta-1}yu\|+\|\mathcal{H}(|y|^{\theta-1}yu)\|)\\
\leq& \ c\|\langle x, y \rangle^{\theta}u\|\\
\leq & \ cM_4.
\end{split}
\end{equation*}

\begin{equation*}
\begin{split}
\|\tilde{Q}_2\|=&\ 4\|w_{N}^{\theta}\mathcal{H}\partial_{x}(xu)\|\\\nonumber
       \leq & \ c(\|\langle x,y \rangle^{\theta-1}\mathcal{H}\partial_{x}(xu)\|+\|\langle x, y \rangle^{\theta-1}\mathcal{H}(x\partial_{x}(xu))\|+\|\langle x,y \rangle^{\theta-1}\mathcal{H}(y \partial_{x}(xu))\|)\\
\leq& \ c(\|\mathcal{H}\partial_{x}(xu)\|+\||x|^{\theta-1}\mathcal{H}\partial_{x}(xu)\|+\||y|^{\theta-1}\mathcal{H}\partial_{x}(xu)\|+\|\mathcal{H}(x\partial_{x}(xu))\|\\
&+\||x|^{\theta-1}\mathcal{H}(x\partial_{x}(xu))\|+\||y|^{\theta-1}\mathcal{H}(x\partial_{x}(xu))\|+\|\mathcal{H}(y\partial_{x}(xu))\|\\
&+\||x|^{\theta-1}\mathcal{H}(y\partial_{x}(xu))\|+\||y|^{\theta-1}\mathcal{H}(y\partial_{x}(xu))\|)\\
\leq & \ c(\|\partial_{x}(xu)\|+c^{*}\||x|^{\theta-1}\partial_{x}(xu)\|+\||y|^{\theta-1}\partial_{x}(xu)\|+\|x\partial_{x}(xu)\|\\
&+c^{*}\||x|^{\theta-1}x\partial_{x}(xu)\|+\||y|^{\theta-1}x\partial_{x}(xu)\|+\|y\partial_{x}(xu)\|+c^{*}\||x|^{\theta-1}y\partial_{x}(xu)\|\\
&+\||y|^{\theta-1}y\partial_{x}(xu)\|)\\
\leq& \ c(M_4 + \|\langle x, y \rangle^{\theta}u\|+\|\langle x, y \rangle^{\theta+1}\partial_{x}u\|)\\
\leq & \ c(M_4 + \|\langle x,y \rangle^{3/2 +\theta}u\|+\|J^{3+2\theta}u\|)\\
\leq & \ cM_4,
\end{split}
\end{equation*}
and
\begin{equation*}
\begin{split}
\|\tilde{E}_2\|\leq& \ \|\langle x,y \rangle^{\theta-1}\mathcal{H}\partial_{x}\partial_{y}^{2}u\|+\|\langle x,y \rangle^{\theta-1}x\mathcal{H}\partial_{x}\partial_{y}^{2}u\|+\|\langle x,y \rangle^{\theta-1}y\mathcal{H}\partial_{x}\partial_{y}^{2}u\|\\
\leq & \ \|\langle x,y \rangle^{\theta-1}\mathcal{H}\partial_{x}\partial_{y}^{2}u\|+\|\langle x,y \rangle^{\theta-1}\mathcal{H}(x\partial_{x}\partial_{y}^{2}u)\|\\
&+\|\langle x,y \rangle^{\theta-1}\mathcal{H}(y\partial_{x}\partial_{y}^{2}u)\|\\
\leq &\ c(\|\mathcal{H}\partial_{x}\partial_{y}^{2}u\|+\||x|^{\theta-1}\mathcal{H}(x\partial_{x}\partial_{y}^{2}u)\|+\||y|^{\theta-1}\mathcal{H}(y\partial_{x}\partial_{y}^{2}u)\|\\
&+\|\mathcal{H}(x\partial_{x}\partial_{y}^{2}u)\|+\||x|^{\theta-1}\mathcal{H}(x\partial_{x}\partial_{y}^{2}u)\|+\||y|^{\theta-1}\mathcal{H}(x\partial_{x}\partial_{y}^{2}u)\|\\
&+\|\mathcal{H}(y\partial_{x}\partial_{y}^{2}u)\|+\||x|^{\theta-1}\mathcal{H}(y\partial_{x}\partial_{y}^{2}u)\|+\||y|^{\theta-1}\mathcal{H}(y\partial_{x}\partial_{y}^{2}u)\|)\\
\leq & \ c\|\langle x,y \rangle^{\theta}\partial_{x}\partial_{y}^{2}u\|\\
\leq & \ M_4+2(\|\langle x,y \rangle^{3/2+\theta}u\|+\|J^{3+2\theta}u\|)\\
\leq& \ cM_4.
\end{split}
\end{equation*}
From this point on, one can proceed as in Case a) and conclude the proof of
Case b). The  proof of Theorem \ref{B1} is thus completed.

\section{Unique continuation principle} \label{uniquep}

This section is devoted to establish Theorems \ref{P1} and \ref{P2}. We
follow closely the arguments in \cite{GermanPonce}. Indeed, the main idea is
to explore the ``bad'' behavior of the BO-ZK in the $x$-direction, which, in
some sense, is similar to the one presented by the BO  equation
\eqref{boequation}. We pointed out that a similar approach was also
successfully applied to the Benjamin equation in \cite{jose}.

\begin{proof}[Proof of Theorem \ref{P1}]
Let us start by noting that the solution of \eqref{bozk} can be represented
by Duhamel's formula
\begin{equation}\label{121}
u(t)=U(t)\phi -\int_{0}^{t}U(t-t')u(t')\partial_{x}u(t')dt',
\end{equation}
where $U(t)\phi$  is  the solution of the IVP associated with the linear
BO-ZK equation. It is easy to check that, via its Fourier transform,
$$
\widehat{U(t)\phi}(\xi,\eta)=e^{it\xi(\eta^2-|\xi|)}\hat{\phi}(\xi,\eta).
$$
Without loss of generality we assume $t_1=0.$ Thus, since $\phi \in
\mathcal{Z}_{5,5/2},$ it follows from Theorem \ref{B1}  that
\begin{equation}\label{H5conc}
u\in C([0,T];H^{5}\cap L^{2}_r), \qquad 0<r<5/2.
\end{equation}

By multiplying \eqref{121} by $|x|^{5/2}$ and then taking the Fourier transform lead to
\begin{equation}\label{aftH5}
D^{1/2}_\xi\partial_\xi^2(\widehat{u(t)})=D^{1/2}_\xi\partial_\xi^2\big(e^{it\xi(\eta^2-|\xi|)}\hat{\phi}\big)
-\int_0^tD^{1/2}_\xi\partial_\xi^2\big(e^{i(t-t')\xi(\eta^2-|\xi|)}\hat{z}\big)dt',
\end{equation}
where $z=\frac{1}{2}\partial_{x}u^{2}.$ Fixed $t\in[0,T]$, remark that if $\langle x,y\rangle^{5/2}U(t)\phi\in
L^2(\R^2)$ then it must be the case that $|x|^{5/2}U(t)\phi\in L^2(\R^2)$,
which, by Plancherel's identity, implies that
$$
D^{1/2}_\xi\partial_\xi^2\big(e^{it\xi(\eta^2-|\xi|)}\hat{\phi}\big)\in
L^2(\R^2).
$$
We will prove that this is possible only if $\hat{\phi}(0,\eta)=0$, for all $\eta\in\R$. The idea goes as follows:  since
\begin{equation}\label{126}
\begin{split}
\partial_{\xi}^{2}(e^{it\xi(\eta^2-|\xi|)}\hat{\phi})=&\ e^{it\xi(\eta^2-|\xi|)}\Big((-2it \mathrm{sgn}(\xi)-4t^{2}\xi^{2}+4t^{2}\eta^{2}|\xi|-t^{2}\eta^{4})\hat{\phi}+(2it\eta^{2}\\
&-4it|\xi|)\partial_{\xi}\hat{\phi}+\partial_{\xi}^{2}\hat{\phi}\Big),
\end{split}
\end{equation}
we will show that all terms in \eqref{aftH5}, except the one involving
$\sgn(\xi)$, arising from the linear part (see \eqref{126}), have the
appropriate decay for all $t\in[0,T]$. This in turn, will imply the desired.

On one hand, in the $x$-direction, the BOZK equation has a similar behavior
as the BO equation, so, following the ideas in \cite{GermanPonce}, we need to
localize in the $\xi$-direction. On the other hand, to control all terms, we
need some strong decay in the $\eta$-direction but not localization. To do
so, define $\chi(\xi,\eta)=\tilde{\chi}(\xi)e^{-\eta^2}$, where $\tilde{\chi}
\in C_{0}^{\infty}(\R)$ is  such that $supp \ \tilde{\chi} \subset
(-\epsilon,\epsilon)$ and $\tilde{\chi} =1$ in $(-\epsilon/2,\epsilon/2).$
Note, in particular, that $\chi\in L^\infty_\eta H^2_\xi$.

With the function $\chi$ in hand, we  write the linear part of Duhamel's
formula as
\begin{equation*}
\begin{split}
\chi D_{\xi}^{1/2}\partial_{\xi}^{2}(e^{it\xi(\eta^2-|\xi|)}\hat{\phi})=&\ [\chi;D_{\xi}^{1/2}]\partial_{\xi}^{2}\Big(e^{it\xi(\eta^2-|\xi|)}\hat{\phi})+
D_{\xi}^{1/2}(\chi\partial_{\xi}^{2}(e^{it\xi(\eta^2-|\xi|)}\hat{\phi})\Big)\\
=&\ A+B.
\end{split}
\end{equation*}
In what follows, the constant $c$ will depend on $T$ and the norms of $\chi$.
From Proposition \ref{C}, Lemma \ref{inter}, Plancherel's identity, and
\eqref{126} it follows that
\begin{equation}\label{teoP1a}
\begin{split}
\|A\|=& \ \|\|[\chi;D_{\xi}^{1/2}]\partial_{\xi}^{2}(e^{it\xi(\eta^2-|\xi|)}\hat{\phi})\|_{L^{2}_{ \xi}}\|_{L^{2}_{\eta}}\\
\leq & \ c\|\|\chi\|_{H^{1}_{\xi}}\|\partial_{\xi}^{2}(e^{it\xi(\eta^2-|\xi|)}\hat{\phi})\|_{L^{2}_{\xi}}\|_{L^{2}_{\eta}}\\
\leq & \ c(\|\hat{\phi}\|+\|\xi^2\hat{\phi}\| +\|\eta^{2}\xi\hat{\phi}\|+\|\eta^{4}\hat{\phi}\|+\|\eta^{2}\partial_{\xi}\hat{\phi}\|+\|\xi\partial_{\xi}\hat{\phi}\| +\|\partial_{\xi}^{2}\hat{\phi}\|)\\
=&\ c(\|\phi\|+\|\partial_{x}^{2}\phi\|+\|\partial_{y}^2
\partial_{x}\phi\|+\|\partial_{y}^{4}\phi\|+\|\partial_{y}^{2}(x\phi)\|+\|\partial_{x}(x\phi)\|+\|x^{2}\phi\|).
\end{split}
\end{equation}
All terms in the right-hand side of \eqref{teoP1a} are finite because
$\phi\in \mathcal{Z}_{4,2}$.

Now write
\begin{equation*}
\begin{split}
B=&\ D^{1/2}_{\xi}\Big((-2it \mathrm{sgn}(\xi)-4t^{2}\xi^{2}+4t^{2}\eta^{2}|\xi|-t^{2}\eta^{4})\hat{\phi}+(2it\eta^{2}\\
&-4it|\xi|)\partial_{\xi}\hat{\phi}+\partial_{\xi}^{2}\hat{\phi}\Big)\\
=&\ B_1 +B_2 +B_3 +B_4+B_5+B_6+B_7.
\end{split}
\end{equation*}
Let us estimate the $L^2$ norm of $B_7$. Theorem \ref{stein}, Proposition
\ref{Pontual}, and Lemma \ref{P} imply
\begin{equation*}
\begin{split}
\|B_7\|=& \ \|D_{\xi}^{1/2}(\chi e^{it\xi(\eta^2-|\xi|)}\partial_{\xi}^{2}\hat{\phi})\|\\
=& \ \|\|D_{\xi}^{1/2}(\chi e^{it\xi(\eta^2-|\xi|)}\partial_{\xi}^{2}\hat{\phi})\|_{L^{2}_{\xi}}\|_{L^{2}_{\eta}}\\
\leq & \ c(\|\|\chi e^{it\xi(\eta^2-|\xi|)}\partial_{\xi}^{2}\hat{\phi}\|_{L^{2}_{\xi}}+\|\mathcal{D}_{\xi}^{1/2}(\chi e^{it\xi(\eta^2-|\xi|)}\partial_{\xi}^{2}\hat{\phi})\|_{L^{2}_{\xi}}\|_{L^{2}_{\eta}})\\
\leq& \ c(\|\chi e^{it\xi(\eta^2-|\xi|)}\partial_{\xi}^{2}\hat{\phi}\|+\|\mathcal{D}_{\xi}^{1/2}(\chi e^{it\xi(\eta^2-|\xi|)}\partial_{\xi}^{2}\hat{\phi})\|)\\
\leq & \ c(\|x^2 \phi\|+\|\mathcal{D}_{\xi}^{1/2}(e^{-it\xi |\xi|})\chi e^{it \xi \eta^{2}}\partial_{\xi}^{2}\hat{\phi}\|+ \|e^{-it\xi |\xi|}\mathcal{D}_{\xi}^{1/2}(\chi e^{it \xi \eta^{2}}\partial_{\xi}^{2}\hat{\phi})\|)\\
\leq & \ c( \|x^2 \phi\|+ \|(t^{1/4}+t^{1/2}|\xi|^{1/2})\chi \partial_{\xi}^{2}\hat{\phi}\|+\|\mathcal{D}_{\xi}^{1/2}(e^{it\xi \eta^{2}})\chi \partial_{\xi}^{2}\hat{\phi}\|\\
&+\|e^{-it\xi \eta^{2}}\mathcal{D}_{\xi}^{1/2}(\chi  \partial_{\xi}^{2}\hat{\phi})\|)\\
\leq& \ c( \|x^2 \phi\|+ \|(t^{1/4}+t^{1/2}|\xi|^{1/2})\chi \|_{\infty}\|\partial_{\xi}^{2}\hat{\phi}\|+\|(\eta^{2}t)^{1/2}\chi \partial_{\xi}^{2}\hat{\phi}\|+\|\mathcal{D}_{\xi}^{1/2}(\chi)\partial_{\xi}^{2}\hat{\phi}\|\\
&+\|\chi\mathcal{D}_{\xi}^{1/2}(\partial_{\xi}^{2}\hat{\phi})\|)\\
\leq &\ c(\|x^2 \phi\|+\|(\eta^{2}t)^{1/2}\chi\|_{\infty}\| \partial_{\xi}^{2}\hat{\phi}\|+\|\mathcal{D}_{\xi}^{1/2}(\chi)\|_{\infty}\|\partial_{\xi}^{2}\hat{\phi}\|+
\|\chi\|_{\infty}\|\mathcal{D}_{\xi}^{1/2}\partial_{\xi}^{2}\hat{\phi}\|)\\
\leq& \ c\|\langle x,y \rangle^{2+1/2}\phi\|.
\end{split}
\end{equation*}
Control on $B_2,B_3,B_4,B_5 \ \mbox{and} \ B_6$ in $L^{2}(\R^2)$ are obtained
in a similar fashion, so we omit the details. Note that we do not estimate
$B_1$. However, if we show that the integral  part of Duhamel's formula is in
$L^2(|x|^5dxdy)$ then we will conclude that $B_1\in L^2(\R^2)$ (for any fixed
$t\in[0,T]$).

To do so, we localize again using the function $\chi$. In fact, using
commutators, the integral part in \eqref{aftH5} reads as
\begin{equation}\label{teoP1b}
\begin{split}
\int_{0}^{t}[\chi;&D_{\xi}^{1/2}]\Big(e^{i(t-t')\xi(\eta^2-|\xi|)}\big(-2i(t-t')\mathrm{sgn}(\xi)\hat{z}
-4(t-t')^{2}\xi^2 \hat{z} +4(t-t')^{2}\eta^2 |\xi|\hat{z}\\
&-(t-t')^{2}\eta^{4}\hat{z}-4(t-t')|\xi|\partial_{\xi}\hat{z}+2i(t-t')\eta^{2}\partial_{\xi}\hat{z}+\partial_{\xi}^{2}\hat{z}\big)\Big)\\
&+D_{\xi}^{1/2}\Big(\chi\big(e^{i(t-t')\xi(\eta^2-|\xi|)}(-2i(t-t')\mathrm{sgn}(\xi)\hat{z}-4(t-t')^{2}\xi^{2}\hat{z}+4(t-t')^{2}\eta^{2}|\xi|\hat{z}\\
&-(t-t')^{2}\eta^{4}\hat{z}-4i(t-t')|\xi|\partial_{\xi}\hat{z}+2i(t-t')\eta^{2}\partial_{\xi}\hat{z}+\partial_{\xi}^{2}\hat{z})\big)\Big)dt'\\
=&\ C_1+...+C_{7}+D_1+...+D_7.
\end{split}
\end{equation}
We observe that the terms involving the highest regularity and decay are
$C_4$ and $D_{7},$ respectively. In the sequel, we show their $L^{2}$
estimates. From Proposition \ref{C}, we obtain
\begin{equation}\label{c4esiteo}
\begin{split}
\|C_4\|\leq & \ t^2\|\|\|\chi\|_{H^{1}_{\xi}}\|e^{i(t-t')\xi(\eta^2-|\xi|)}(t-t')^{2}\eta^4 \hat{z}\|_{L^{2}_{\xi}}\|_{L^{2}_{\eta}}\|_{L^{1}_{T}}\\
\leq & \ c\|\|\partial_{y}^{4}z\|\|_{L^{1}_{T}}\\
\leq & \ c \|\|\partial_{y}^{4}\partial_{x}u^2\|\|_{L^{1}_{T}}\\
\leq& \ c\|u\|_{L^{\infty}_{T} H^{5}}^{2}.
\end{split}
\end{equation}
The right-hand side of \eqref{c4esiteo} in finite thanks to \eqref{H5conc}.

Regarding the $L^2$ norm of $D_7$, first observe that using Lemma
\ref{inter}, we deduce that
\begin{equation}\label{teoP1c}
xu\in H^{2}(\R^2) \ \mbox{and} \ \||x|^{3/2}\partial_{x}u\|\leq
c(\|u\|_{\mathcal{Z}_{4,2}}+\|\langle x,y \rangle^{1/2}u\|).
\end{equation}
Let $\bar{D}_{7}=D^{1/2}_{\xi}(\chi
e^{i(t-t')\xi(\eta^{2}-|\xi|)})\partial_{\xi}^{2}\hat{z}$. Theorem
\ref{stein}, Proposition \ref{Pontual}, Lemma \ref{P}, and \eqref{teoP1c}
yield
\begin{equation*}
\begin{split}
\|\bar{D}_{7}\|=&\ \|D^{1/2}_{\xi}(\chi e^{i(t-t')\xi(\eta^{2}-|\xi|)}\partial_{\xi}^{2}\hat{z})\|\\
=&\ \|\|D^{1/2}_{\xi}(\chi e^{i(t-t')\xi(\eta^{2}-|\xi|)}\partial_{\xi}^{2}\hat{z})\|_{L^{2}_{\xi}}\|_{L^{2}_{\eta}}\\
 \leq & \ \|\|\chi e^{i(t-t')\xi(\eta^2-|\xi|)}\partial_{\xi}^{2}\hat{z}\|_{L^{2}_{\xi}}+\|\mathcal{D}_{\xi}^{1/2}(\chi e^{i(t-t')\xi(\eta^2-|\xi|)}\partial_{\xi}^{2}\hat{z})\|_{L^{2}_{\xi}}\|_{L^{2}_{\eta}}\\
\leq& \ c(\|\chi e^{i(t-t')\xi(\eta^2-|\xi|)}\partial_{\xi}^{2}\hat{z}\|+\|\mathcal{D}_{\xi}^{1/2}(\chi e^{i(t-t')\xi(\eta^2-|\xi|)}\partial_{\xi}^{2}\hat{z})\|)\\
\leq & \ c(\|x^{2}z\|+\|\mathcal{D}_{\xi}^{1/2}(e^{-i(t-t')\xi|\xi|})\chi e^{i(t-t')\xi \eta^{2}}\|\\
&+\|e^{-i(t-t')\xi|\xi|}\mathcal{D}_{\xi}^{1/2}(\chi e^{i(t-t')\xi \eta^{2}} \partial_{\xi}^{2}\hat{z})\|)\\
\leq& \ c(\|x^{2}z\|+\|\chi(T^{1/4}+T^{1/2}|\xi|^{1/2})\|_{\infty}\|\partial_{\xi}^{2}\hat{z}\|+\|\mathcal{D}_{\xi}^{1/2}(e^{i(t-t')\xi \eta^{2}})\chi \partial_{\xi}^{2}\hat{z}\|\\
&+\|e^{i(t-t')\xi \eta^{2}}\mathcal{D}_{\xi}^{1/2}(\chi \partial_{\xi}^{2}\hat{z})\|)\\
\leq& \ c(\|x^{2}z\|+\|(\eta^{2}t)^{1/2}\chi\|_{\infty}\|\partial_{\xi}^{2}\hat{z}\|+\|\mathcal{D}_{\xi}^{1/2}(\chi)\|_{\infty}\|\partial_{\xi}^{2}\hat{z}\|+
\|\chi\|_{\infty}\|\mathcal{D}_{\xi}^{1/2}\partial_{\xi}^{2}\hat{z}\|)\\
\leq& \ c(\|x^{2}z\|+\|D_{\xi}^{1/2}\partial_{\xi}^{2}\hat{z}\|)\\
=&\ c(\|x^{2}z\|+\||x|^{2+1/2}z\|)\\
\leq& \ c(\|x^{2}uu_{x}\|+\||x|^{2+1/2}uu_{x}\|)\\
\leq&\ c(\|u_{x}\|_{\infty}\|x^{2}u\|+\||x|^{3/2}\partial_{x}u\|\|xu\|_{L^{\infty}})\\
\leq&\ c(\|u_{x}\|_{\infty}\|x^{2}u\|+\|xu\|_{\infty}\|u\|_{\mathcal{Z}_{4,2}}+\|\langle x,y \rangle^{1/2}u\|).
\end{split}
\end{equation*}
 As a consequence,
 $$
 \|D_7\|\leq c\|\|\bar{D}_{7}\|\|_{L^{1}_T}<\infty.
 $$
Since $\hat{z}(0)=0,$ we can estimate $D_1$ as follows. First, we notice that
\begin{equation}\label{d1}
\begin{split}
\|D_{\xi}^{1/2}(\sgn(\xi)\hat{z})\|=& \ \|\widehat{|x|^{1/2}\mathcal{H}z}\|\\
                                \leq & \ \|(1+|x|)^{1/2}\mathcal{H}z\|\\
                                \leq & \ \|(1+|x|)\mathcal{H}z\|\\
                                \leq & \ \|z\|+\|x\mathcal{H}z\|\\
                                \leq & \ \|z\|+\|\mathcal{H}(xz)\|\\
                                =& \ \|z\|+\|xz\|\\
                                \leq & \ c\|u\|_{\mathcal{Z}_{4,2}}.
\end{split}
\end{equation}
Let $$\bar{D}_1=D_{\xi}^{1/2}\Big(\chi
\big(e^{i(t-t')\xi(\eta^2-|\xi|)}((t-t')\mathrm{sgn}(\xi)\hat{z})\big)\Big).$$
Then Theorem \ref{stein}, Proposition \ref{Pontual}, Lemma \ref{P}, and
\eqref{d1} yield
\begin{equation*}
\begin{split}
\|\bar{D}_{1}\|=&\ \|D^{1/2}_{\xi}(\chi e^{i(t-t')\xi(\eta^{2}-|\xi|)}\sgn(\xi)\hat{z})\|\\
=&\ \|\|D^{1/2}_{\xi}(\chi e^{i(t-t')\xi(\eta^{2}-|\xi|)}\sgn(\xi)\hat{z})\|_{L^{2}_{\xi}}\|_{L^{2}_{\eta}}\\
 \leq & \ \|\|\chi e^{i(t-t')\xi(\eta^2-|\xi|)}\sgn(\xi)\hat{z}\|_{L^{2}_{\xi}}+\|\mathcal{D}_{\xi}^{1/2}(\chi e^{i(t-t')\xi(\eta^2-|\xi|)}\sgn(\xi)\hat{z})\|_{L^{2}_{\xi}}\|_{L^{2}_{\eta}}\\
\leq& \ c(\|\chi e^{i(t-t')\xi(\eta^2-|\xi|)}\sgn(\xi)\hat{z}\|+\|\mathcal{D}_{\xi}^{1/2}(\chi e^{i(t-t')\xi(\eta^2-|\xi|)}\sgn(\xi)\hat{z})\|)\\
\leq & \ c(\|z\|+\|\mathcal{D}_{\xi}^{1/2}(e^{-i(t-t')\xi|\xi|})\chi e^{i(t-t')\xi \eta^{2}}\sgn(\xi)\hat{z}\|\\
&+\|e^{-i(t-t')\xi|\xi|}\mathcal{D}_{\xi}^{1/2}(\chi e^{i(t-t')\xi \eta^{2}} \sgn(\xi)\hat{z})\|)\\
\leq& \ c(\|z\|+\|\chi(T^{1/4}+T^{1/2}|\xi|^{1/2})\|_{\infty}\|\sgn(\xi)\hat{z}\|+\|\mathcal{D}_{\xi}^{1/2}(e^{i(t-t')\xi \eta^{2}})\chi \sgn(\xi)\hat{z}\|\\
&+\|e^{i(t-t')\xi \eta^{2}}\mathcal{D}_{\xi}^{1/2}(\chi \sgn(\xi)\hat{z})\|)\\
\leq& \
c(\|z\|+\|(\eta^{2}t)^{1/2}\chi\|_{\infty}\|\sgn(\xi)\hat{z}\|+\|\mathcal{D}_{\xi}^{1/2}(\chi)\|_{\infty}\|\sgn(\xi)\hat{z}\|\\
&+
\|\chi\|_{\infty}\|\mathcal{D}_{\xi}^{1/2}(\sgn(\xi)\hat{z})\|)\\
\leq& \ c(\|z\|+\|D_{\xi}^{1/2}(\sgn(\xi)\hat{z})\|)\\
\leq & c\|u\|_{\mathcal{Z}_{4,2}}.
\end{split}
\end{equation*}
Therefore,
$$\|D_1\|\leq \|\bar{D}_1\|_{L^{1}_{T}}<\infty.$$

The other terms appearing in \eqref{teoP1b} are estimated in a very similar
manner.  Here, we also omit the details. Hence, the above estimates on the
linear and integral parts of \eqref{aftH5}, together with the  fact that
$u(t_2)\in \mathcal{Z}_{5,5/2}$, lead to concluding that
$$
B_1=-2it_2D^{1/2}_{\xi}(\chi
e^{it_2\xi(\eta^2-|\xi|)}\mathrm{sgn}(\xi)\hat{\phi})\in L^{2}(\R^2).
$$
Fubini's theorem then gives that $B_1\in L^{2}_{\xi}(\R), \mathrm{a.e.} \ \eta \in \R.$
So, in view of Theorem \ref{stein}, we deduce
\begin{equation}\label{teoP1e}
\mathcal{D}^{1/2}_{\xi}(\chi
e^{it_2\xi(\eta^2-|\xi|)}\mathrm{sgn}(\xi)\hat{\phi})\in
L_{\xi}^{2}(\R),\quad \mathrm{a.e.} \ \eta \in \R.
\end{equation}
An application of Proposition \ref{localint} gives
$$
\hat{\phi}(0,\eta)=0, \ \mathrm{a.e.} \  \eta \in \R.
$$
Since $\hat{\phi}$ is continuous we obtain $\hat{\phi}(0,\eta)=0$, for all
$\eta \in \R$. The conclusion of the theorem follows just taking a look at
\eqref{fourieru}.
\end{proof}

\begin{proof}[Proof of Theorem \ref{P2}]
Without loss of generality  we assume $t_1=0<t_2<t_3,$ and explore the
arguments in \cite{GermanPonce}. Indeed, by multiplying \eqref{121} by
$|x|^{7/2}$ and taking the Fourier transform, we obtain
\begin{equation}\label{teop2a}
D_{\xi}^{1/2}\partial_{\xi}^3\widehat{u(t)}=D_{\xi}^{1/2}F(t,\xi,\eta,\hat{\phi})-\int_{0}^{t}D_{\xi}^{1/2}F(t-t',\xi,\eta,\hat{z}(t'))dt',
\end{equation}
where
$F(t,\xi,\eta,\hat{\phi})=\partial_{\xi}^{3}(e^{it\xi(\eta^2-|\xi|)}\hat{\phi})$.
Thus, Plancherel's theorem evince that if we assume  that the right-hand side
of \eqref{teop2a} belongs to $L^{2}(\R^2)$, for times $t_1=0<t_2<t_3,$  then
we will obtain a contradiction. Here, as before
$z=\frac{1}{2}\partial_{x}u^2$.

First of all, we note that our assumptions together with Theorems \ref{B1}
and \ref{P1} implies that
$$
u\in C([0,T];\mathcal{\dot{Z}}_{s,r}), \qquad \frac{5}{2}\leq r<\frac{7}{2}.
$$
Moreover, a straightforward  computation reveals that
\begin{equation}\label{comput}
\begin{split}
\partial_{\xi}^{3}(e^{it\xi(\eta^2-|\xi|)}\hat{\phi})=&\ \Big((-4it\delta_{\xi}-it^{3}\eta^{6}-24t^{2} \xi +6t^{2} \eta^{2} \mathrm{sgn}(\xi)+8i t^{3} |\xi|^3  \\
& +6it^3 \eta^4 |\xi|-12it^3 \eta^2 \xi^2)\hat{\phi}+(-6it \mathrm{sgn}(\xi)-12t^2 \xi^2 +12 t^2 \eta^2 |\xi|\\
&-3t^2 \eta^4 )\partial_{\xi}\hat{\phi} + 3it(\eta^2 -
2|\xi|)\partial_{\xi}^{2}\hat{\phi}+\partial_{\xi}^{3}\hat{\phi}\Big)e^{it\xi(\eta^2-|\xi|)}.
\end{split}
\end{equation}
Here, $\delta_\xi$ stands for the Dirac delta function with respect to $\xi$,
that is, $\langle\delta_\xi,\varphi\rangle=\varphi(0,\eta)$, for all
$\varphi\in\mathcal{S}(\R^2)$.

The proof follows closely the arguments in Theorem \ref{P1}. Recall that
$\chi(\xi,\eta)=\tilde{\chi}(\xi)e^{-\eta^2}$, where $\tilde{\chi} \in
C_{0}^{\infty}(\R),$  $supp \ \tilde{\chi} \subset (-\epsilon,\epsilon)$ and
$\tilde{\chi} =1$ in $(-\epsilon/2,\epsilon/2).$ Hence, we may write
\begin{equation*}
\begin{split}
\chi D_{\xi}^{1/2}\partial_{\xi}^{3}(e^{it\xi(\eta^2-|\xi|)}\hat{\phi})=&\ [\chi;D_{\xi}^{1/2}]\partial_{\xi}^{3}(e^{it\xi(\eta^2-|\xi|)}\hat{\phi})+D_{\xi}^{1/2}(\chi\partial_{\xi}^{3}(e^{it\xi(\eta^2-|\xi|)}\hat{\phi}))\\
=& \ \tilde{A}+\tilde{B}.
\end{split}
\end{equation*}
To estimate the $L^2$ norm of $\tilde{A}$, we can proceed in a very similar
way to its counterpart $A$ in Theorem \ref{P1}. So, we omit the details.

Next, we observe that
\begin{equation}\label{compu}
\begin{split}
\tilde{B}=&D_{\xi}^{1/2}(\chi \partial_{\xi}^{3}(e^{it\xi(\eta^2-|\xi|)}\hat{\phi}))\\
=&\ \chi e^{it\xi(\eta^2-|\xi|)}\Big((-4it\delta_{\xi}-it^{3}\eta^{6}-24t^{2} \xi +6t^{2} \eta^{2} \mathrm{sgn}(\xi)+8i t^{3} |\xi|^3  \\
& +6it^3 \eta^4 |\xi|-12it^3 \eta^2 \xi^2)\hat{\phi}+(-6it \mathrm{sgn}(\xi)-12t^2 \xi^2 +12 t^2 \eta^2 |\xi|\\
&-3t^2 \eta^4 )\partial_{\xi}\hat{\phi} + 3it(\eta^2 -
2|\xi|)\partial_{\xi}^{2}\hat{\phi}+\partial_{\xi}^{3}\hat{\phi}\Big)\\
=&\tilde{B}_2+...+\tilde{B}_{14}.
\end{split}
\end{equation}
From our assumptions, Theorem \ref{P1} implies that the initial data  $\phi$
also belongs to $\dot{\mathcal{Z}}_{5,5/2}$. Thus, the first term involving
the Dirac function in \eqref{compu} must vanishes, that is, the term
$\tilde{B}_1$ does not appear in \eqref{compu}. To estimate $\tilde{B}_4$ we
use that $\hat{\phi}(0,\eta)=0.$ For shortness, we will estimate in details
only the most difficult terms, that is,  the terms $\tilde{B}_{2}$ and
$\tilde{B}_{14}$ which are the ones involving the highest regularity and
decay of the initial data. The other terms, except $\tilde{B}_{8}$, can be
estimated in a similar way.

From Theorem \ref{stein}, \eqref{Leib}, Proposition \ref{Pontual}, Lemma \ref{P},
and H\"older's inequality it follows that
\begin{equation*}
\begin{split}
\|\tilde{B}_2\|\leq& \ c(\|\chi e^{it\xi(\eta^2-|\xi|)} \eta^6 \hat{\phi}\|+\|\mathcal{D}_{\xi}^{1/2}(\chi e^{it\xi(\eta^2-|\xi|)}\eta^6 \hat{\phi})\|)\\
\leq& \ c(\|\phi\|+\|\mathcal{D}_{\xi}^{1/2}(e^{-it|\xi|\xi})\chi e^{it\xi \eta^2}\eta^6 \hat{\phi}\|+\|e^{-it\xi |\xi|}\mathcal{D}_{\xi}^{1/2}(\chi \eta^6 \hat{\phi})\|)\\
\leq& \ c(\|\phi\|+\|\mathcal{D}_{\xi}^{1/2}(e^{it\xi \eta^2})\chi \eta^6 \hat{\phi}\|+\|e^{it\xi \eta^2}\mathcal{D}_{\xi}^{1/2}(\chi e^{it\xi \eta^2} \eta^6 \hat{\phi})\|)\\
\leq& \ c(\|\phi\|+\|\mathcal{D}_{\xi}^{1/2}(\chi \eta^6)\hat{\phi}\|+\|\chi \eta^6 \mathcal{D}_{\xi}^{1/2}\hat{\phi}\|)\\
\leq & \ c(\|\phi\|+\|\mathcal{D}_{\xi}^{1/2}(\chi \eta^6)\|_{\infty}\|\hat{\phi}\|+\|\chi \eta^6\|_{\infty}\|\mathcal{D}_{\xi}^{1/2} \hat{\phi}\|)\\
\leq& \ c(\|\phi\|+\|D_{\xi}^{1/2}\hat{\phi}\|)\\
=&\ c(\|\phi\|+\||x|^{1/2}\phi\|).
\end{split}
\end{equation*}
Similarly,
\begin{equation*}
\begin{split}
\|\tilde{B}_{14}\|\leq& \ c(\|\chi e^{it\xi(\eta^2-|\xi|)} \partial_{\xi}^{3}\hat{\phi}\|+\|\mathcal{D}_{\xi}^{1/2}(\chi e^{it\xi(\eta^2-|\xi|)}\partial_{\xi}^{3}\hat{\phi})\|)\\
\leq& \ c(\|x^3 \phi\|+\|\mathcal{D}_{\xi}^{1/2}(e^{-it\xi |\xi|})\chi e^{it\xi \eta^2}\partial_{\xi}^3\hat{\phi}\|+\|e^{-it\xi |\xi|}\mathcal{D}_{\xi}^{1/2}(\chi e^{it\xi \eta^2}\partial_{\xi}^{3}\hat{\phi})\|)\\
\leq& \ c(\|x^3 \phi\|+\|\mathcal{D}_{\xi}^{1/2}(e^{it\xi \eta^2})\chi \partial_{\xi}^3 \hat{\phi}\|+\|e^{i\xi \eta^2}\mathcal{D}_{\xi}^{1/2}(\chi \partial_{\xi}^{3}\hat{\phi})\|)\\
\leq& \ c(\|x^3 \phi\|+\|\mathcal{D}_{\xi}^{1/2} \chi\|_{\infty}\|\partial_{\xi}^{3}\hat{\phi}\|+\|\chi\|_{\infty}\|D_{\xi}^{1/2}\partial_{\xi}^{3}\hat{\phi}\|)\\
\leq& \ c\|\langle x, y \rangle^{3+1/2}\phi\|.
\end{split}
\end{equation*}

Now, looking at the integral part we localize again near the origin in Fourier space and use a commutator to get
\begin{equation}\label{2005}
\begin{split}
\int_{0}^{t}[\chi;&D_{\xi}^{1/2}]\Big(\Big\{e^{i(t-t')\xi(\eta^2-|\xi|)}\Big[\big(-4i(t-t')\delta_{\xi}-i(t-t')^{3}\eta^{6}
 -24(t-t')^2 \xi \\
&+6(t-t')^{2}\eta^{2}\mathrm{sgn}(\xi)
+8i (t-t')^{3} |\xi|^3 +6i(t-t')^3 \eta^4 |\xi|
\\
&-12i(t-t')^3 \eta^2 \xi^2\big)\hat{z}+(-6i(t-t') \mathrm{sgn}(\xi)-12(t-t')^2 \xi^2 +12 (t-t')^2 \eta^2 |\xi| \\
&-3(t-t')^2 \eta^4 )\partial_{\xi}\hat{z} +3i(t-t')(\eta^2 - 2|\xi|)\partial_{\xi}^{2}\hat{z}+\partial_{\xi}^{3}\hat{z}\Big]\Big\}\\
&+D_{\xi}^{1/2}\Big\{\chi
e^{i(t-t')\xi(\eta^2-|\xi|)}\Big[\big(-4i(t-t')\delta_{\xi}-i(t-t')^{3}\eta^{6}
 -24(t-t')^2 \xi \\
&+6(t-t')^{2}\eta^{2}\mathrm{sgn}(\xi)
+8i (t-t')^{3} |\xi|^3 +6i(t-t')^3 \eta^4 |\xi|
\\
&-12i(t-t')^3 \eta^2 \xi^2\big)\hat{z}+(-6i(t-t') \mathrm{sgn}(\xi)-12(t-t')^2 \xi^2 +12 (t-t')^2 \eta^2 |\xi| \\
&-3(t-t')^2 \eta^4 )\partial_{\xi}\hat{z} +3i(t-t')(\eta^2 - 2|\xi|)\partial_{\xi}^{2}\hat{z}+\partial_{\xi}^{3}\hat{z}\Big]\Big\}\Big)dt'\\
=& \ \tilde{C}_1+...\tilde{C}_{14}+\tilde{D}_1 +...+\tilde{D}_{13}+\tilde{E},
\end{split}
\end{equation}
where
$$
\tilde{E}=-6i\int_{0}^{t}D_{\xi}^{1/2}(e^{i(t-t')\xi(\eta^2-|\xi|)}\chi (t-t')\mathrm{sgn}(\xi)\partial_{\xi}\hat{z})dt'.
$$
From $\hat{z}(0,\eta,t')=0$ we deduce that $\tilde{C}_1=0$ and
$\tilde{D}_1=0.$ The estimates for the terms
$\tilde{C}_2,...,\tilde{C}_{14}$, $\tilde{D}_2,...,\tilde{D}_{13}$ are
essentially the same ones as those for $C_1,...,C_7$, $D_1,...,D_7$ in
Theorem \ref{P1}. For example, to estimate the term $\tilde{C}_2$ we can use
Proposition \ref{C} to obtain
\begin{equation*}
\begin{split}
\|\tilde{C}_2\|\leq & \ t^3\|\|\|\chi\|_{H^{1}_{\xi}}\|e^{it\xi(\eta^2-|\xi|)}(t-t')^{2}\eta^6 \hat{z}\|_{L^{2}_{\xi}}\|_{L^{2}_{\eta}}\|_{L^{1}_{T}}\\
\leq & \ c\|\|\partial_{y}^{6}z\|\|_{L^{1}_{T}}\\
\leq & \ c \|\|\partial_{y}^{6}\partial_{x}u^2\|\|_{L^{1}_{T}}\\
\leq& \ c\|u\|_{L^{\infty}_{T} H^{7}}^{2}.
\end{split}
\end{equation*}
Let $t=t_2$, since $\phi, u(t_2)\in Z_{7,7/2}$, from  \eqref{2005} and the
previous estimates we conclude that
\begin{align*}
R(t)=&\ \tilde{B}_8 - \tilde{E}\\
=&\ 6i\int_{0}^{t}D_{\xi}^{1/2}(e^{i(t-t')\xi(\eta^2-|\xi|)}\chi (t-t')\mathrm{sgn}(\xi)\partial_{\xi}(\frac{i\xi}{2}\hat{u}\ast \hat{u}))dt'\\
&-6iD_{\xi}^{1/2}(e^{it\xi(\eta^2-|\xi|)}\chi t \mathrm{sgn}(\xi)\partial_{\xi}\hat{\phi})\\
=&\ 6i\int_{0}^{t}D_{\xi}^{1/2}(e^{i(t-t')\xi(\eta^2-|\xi|)}\chi (t-t')\mathrm{sgn}(\xi)(\partial_{\xi}\hat{z}(\xi,\eta,t')
-\partial_{\xi}\hat{z}(0,\eta,t')))dt'\\
&-6iD_{\xi}^{1/2}(e^{it\xi(\eta^2-|\xi|)} \chi t \mathrm{sgn}(\xi)(\partial_{\xi}\hat{\phi}(\xi,\eta)-\partial_{\xi}\hat{\phi}(0,\eta)))\\
&-6iD_{\xi}^{1/2}(e^{it\xi(\eta^2-|\xi|)} \chi t \mathrm{sgn}(\xi)\partial_{\xi}\hat{\phi}(0,\eta))\\
&+6i\int_{0}^{t}D_{\xi}^{1/2}(e^{i(t-t')\xi(\eta^2-|\xi|)}\chi (t-t')\mathrm{sgn}(\xi)\partial_{\xi}\hat{z}(0,\eta,t'))dt'\\
=&\ R_{1}(t)+R_{2}(t)+R_{3}(t)+R_{4}(t)\in L^{2}(\R^2).
\end{align*}
Let us check that $R_1 (t)\in L^{2}(\R^2)$. Indeed, let
$$ f(\xi,\eta,t')=e^{i(t-t')\xi(\eta^2-|\xi|)}\chi
(t-t')\mathrm{sgn}(\xi)g(\xi,\eta,t'),$$ where
$g(\xi,\eta,t')=\partial_{\xi}\hat{z}(\xi,\eta,t')-\partial_{\xi}\hat{z}(0,\eta,t').$
A simple computation gives us
$$
g=\frac{i}{2}(\widehat{u^2}+\xi \partial_{\xi}\widehat{u^2}-\widehat{u^2}(0,\eta,t')),\qquad
\partial_{\xi}g= i\partial_{\xi}\widehat{u^2}+\frac{i}{2}\xi \partial_{\xi}^2\widehat{u^2},
$$
and
$$\partial_{\eta}g=\frac{i}{2}(\partial_{\eta}\widehat{u^2}+\xi \partial_{\eta}\partial_{\xi}\widehat{u^2}-\partial_{\eta}\widehat{u^2}(0,\eta,t')).$$
Since $u^2 \in C([0,T];Z_{7,\frac{7}{2}-\epsilon})$, for any $0<\epsilon
<1/2$, it follows that $\widehat{u^2}\in
C([0,T];Z_{\frac{7}{2}-\epsilon,7}).$ Then, the identity $g\delta_\xi=0$ and
Sobolev embedding yield
$$\|f\|\leq c(\|\chi\|\|\widehat{u^2}\|_{\infty}+\|\xi \chi\|\|\partial_{\xi}\widehat{u^2}\|_{\infty})<\infty,$$
\begin{equation*}
\begin{split}
\|\partial_{\xi}f\|\leq& \ c(\|(\eta^2 +2|\xi|)\chi\|\|\widehat{u^2}\|_{\infty}+\|(\eta^2 +2|\xi|)\xi \chi\|\|\partial_{\xi}\widehat{u^2}\|_{\infty}+\|\partial_{\xi}\chi\|\|\widehat{u^2}\|_{\infty}\\
&+\|\xi
\partial_{\xi}\chi\|\|\partial_{\xi}\widehat{u^2}\|_{\infty}+\|\chi\|\|\partial_{\xi}\widehat{u^2}\|_{\infty}+\|\xi
\chi\|\|\partial_{\xi}^{2}\widehat{u^2}\|_{\infty})<\infty,
\end{split}
\end{equation*}
and
\begin{equation*}
\begin{split}
\|\partial_{\eta}f\|\leq& \ c(\|\xi\eta
\chi\|\|\widehat{u^2}\|_{\infty}+\|\xi^2 \eta
\chi\|\|\partial_{\xi}\widehat{u^2}\|_{\infty}+\|\partial_{\eta}\chi\|\|\widehat{u^2}\|_{\infty}
+\|\xi\partial_{\eta}\chi\|\|\partial_{\xi}\widehat{u^2}\|_{\infty}\\
&+\|\xi\partial_{\eta}\chi\|\|\partial_{\xi}\widehat{u^2}\|_{\infty}+
\|\chi\|\|\partial_{\eta}\widehat{u^2}\|_{\infty}+\|\xi
\chi\|\|\partial_{\eta}\partial_{\xi}\widehat{u^2}\|_{\infty})<\infty.
\end{split}
\end{equation*}
Therefore $f(\cdot,\cdot,t')\in H^{1}(\R^2),$ for all $t'\in [0,t].$ It is
now an easy consequence to show that $D_{\xi}^{1/2}f\in
C([0,T];L^{2}(\R^2)).$

A similarly analysis leads to $R_{2}(t)\in L^{2}(\R^2)$. Therefore $R_3 +R_4
\in L^{2}(\R^2).$

Note that
$$
\partial_{\xi}(\frac{i\xi}{2}\hat{u}\ast \hat{u})(0,\eta,t')=\frac{i}{2}\int e^{-i\eta
y}u^{2}(x,y,t')dxdy.
$$
Also, from \eqref{bozk}, we get
\begin{equation}\label{124}
\frac{d}{dt'}\int xe^{-i\eta y}u(x,y,t')dxdy=\frac{1}{2}\int e^{-i\eta
y}u^{2}(x,y,t')dxdy, \forall \eta \in \R,
\end{equation}
which implies the identity
\begin{equation}\label{131}
\partial_{\xi}(\frac{i\xi}{2}\hat{u}\ast \hat{u})(0,\eta,t')=i\frac{d}{dt'}\int xe^{-i\eta y}u(x,y,t')dxdy.
\end{equation}
Substituting \eqref{131} into $R_4$ and integrating by parts, we obtain
\begin{equation*}
\begin{split}
R_4(t)=&\ 6i\int_{0}^{t}D_{\xi}^{1/2}(e^{i(t-t')\xi(\eta^2-|\xi|)}\chi (t-t')\mathrm{sgn}(\xi))(i\frac{d}{dt'}\int xe^{-i\eta y}u(x,y,t')dxdy)dt'\\
=&\ -6D_{\xi}^{1/2}(e^{i(t-t')\xi(\eta^2-|\xi|)}\chi (t-t')\mathrm{sgn}(\xi))\int xe^{-i\eta y}u(x,y,t')|_{t'=0}^{t'=t}+\\
&+6\int_{0}^{t}D_{\xi}^{1/2}((i\xi|\xi|-i\xi\eta^2)e^{i(t-t')\xi(\eta^2-|\xi|)}\chi (t-t')\mathrm{sgn}(\xi))\int xe^{-i\eta y}udxdydt'\\
&-6\int_{0}^{t}D_{\xi}^{1/2}(e^{i(t-t')\xi(\eta^2-|\xi|)}\chi \mathrm{sgn}(\xi)\int xe^{-i\eta y}u(x,y,t')dxdy)dt'\\
=&\ 6D_{\xi}^{1/2}(e^{it\xi(\eta^2-|\xi|)}\chi t \mathrm{sgn}(\xi)\int xe^{-i\eta y}\phi(x,y)dxdy)\\
&+6\int_{0}^{t}D_{\xi}^{1/2}((i\xi |\xi|-i\xi \eta^2)e^{i(t-t')\xi(\eta^2-|\xi|)}\chi \mathrm{sgn}(\xi))\int x e^{-i\eta y}u(x,y,t')dxdydt'\\
&-6\int_{0}^{t}D_{\xi}^{1/2}(e^{i(t-t')\xi(\eta^2-|\xi|)}\chi \mathrm{sgn}(\xi))\int xe^{-i\eta y} u(x,y,t')dxdydt'\\
=&\ -R_3 +R_5+R_6,
\end{split}
\end{equation*}
where above we used the identity
$$
\partial_{\xi}\hat{\phi}(0,\eta)=-i\int xe^{-i\eta y}\phi(x,y)dxdy.
$$
Thus,
$$
R=R_1 +R_2 +R_5+R_6.
$$
Similarly to $R_{1}(t)$ we can show that $R_5\in L^{2}(\R^2)$. Hence,
\begin{equation*}
\begin{split}
R_6(t)=& \ 6\int_{0}^{t}D_{\xi}^{1/2}\Big(e^{i(t-t')\xi(\eta^2-|\xi|)}\chi \mathrm{sgn}(\xi)\int xe^{-i\eta y} u(x,y,t')dxdy\Big)dt'\\
=&\ 6D_{\xi}^{1/2}\int_{0}^{t}\Big(e^{i(t-t')\xi(\eta^2-|\xi|)}\chi \mathrm{sgn}(\xi)\int xe^{-i\eta y} u(x,y,t')dxdy\Big)dt'\in L^{2}(\R^2).
\end{split}
\end{equation*}
Fubini's theorem gives $R_6(t)\in L^{2}_{\xi}(\R),$ $\mathrm{a.e.}$
$\eta \in \R.$
Theorem \ref{stein} then yields
\begin{equation*}
\mathcal{D}_{\xi}^{1/2}\left(\chi \mathrm{sgn}(\xi)\int_{0}^{t}
\big(e^{i(t-t')\xi(\eta^2-|\xi|)}\int xe^{-i\eta y}
u(x,y,t')dxdy\big)dt'\right)\in L^{2}_{\xi}(\R), \mathrm{a.e.} \eta\in \R,
\end{equation*}
 which from Proposition \ref{localint} implies that
\begin{equation*}
0=\int_{0}^{t_2}\left(\int xe^{-i\eta y}u(x,y,t')dxdy\right)dt'=:g(\eta) \
\mathrm{a.e.} \ \eta\in \R.
\end{equation*}
Since $g$ is a continuous function
\begin{equation*}
g(0)=\int_{0}^{t_2}\int xu(x,y,t')dxdydt'=0.
\end{equation*}
 By  Rolle's lemma, there exists $\tau_1\in (0,t_2)$ such that
\begin{equation}\label{123}
\int xu(x,y,\tau_1)dxdy=0.
\end{equation}
Analogously, using that \ $u(t_2),u(t_3)\in Z_{7,7/2}$ \ we can show the
existence of $\tau_2 \in (t_2,t_3)$ such that
\begin{equation}\label{122}
\int xu(x,y,\tau_2)dxdy=0.
\end{equation}
Finally, from \eqref{123}, \eqref{122}, \eqref{124} (with $\eta=0)$, and the
fact that the $L^2$ norm of $u$ is conserved, we conclude $\|\phi\|=0$. The
uniqueness os solutions  then implies the desired.

Thus, we complete the proof of the theorem.
\end{proof}

\section*{Acknowledgement}

A.P. is partially supported by CNPq-Brazil under grant 301535/2010-8.


\bibliographystyle{mrl}

\end{document}